\documentclass[11pt]{amsart}
\usepackage{amsfonts}
\usepackage{mathrsfs}
\usepackage{amsmath,amssymb,amsxtra}
\usepackage{amssymb}
\usepackage{hyperref}
\usepackage{hypcap}
\usepackage{bookmark}
\usepackage{amsthm}
\usepackage{algorithm}
\usepackage{algorithmic}
\usepackage{array}
\usepackage{xcolor}
\allowdisplaybreaks

\numberwithin{equation}{section}

\def\Z{\mathbb{Z}{\ssc\,}}

\def\C{\mathbb{C}{\ssc\,}}

\def\ssc{\scriptscriptstyle}

\def\d{\delta}
\def\D{\Delta}

\def \End{{\rm End}}

\def \<{\langle}
\def \>{\rangle}

\def\gl{\mathfrak {gl}}
\def\sl{\mathfrak {sl}}

\def\vs{\vspace*}

\def \be{\begin{equation}\label}
\def \ee{\end{equation}}
\def \bex{\begin{example}\label}
\def \eex{\end{example}}
\def \bl{\begin{lem}\label}
\def \el{\end{lem}}
\def \bt{\begin{thm}\label}
\def \et{\end{thm}}
\def \bp{\begin{prop}\label}
\def \ep{\end{prop}}
\def \br{\begin{rem}\label}
\def \er{\end{rem}}
\def \bc{\begin{coro}\label}
\def \ec{\end{coro}}
\def \bd{\begin{de}\label}
\def \ed{\end{de}}

\newtheorem{thm}{Theorem}[section]
\newtheorem{prop}[thm]{Proposition}
\newtheorem{coro}[thm]{Corollary}

\newtheorem{lem}[thm]{Lemma}

\newtheorem{de}[thm]{Definition}

\theoremstyle{definition}
\newtheorem{example}[thm]{Example}
\newtheorem{rem}[thm]{Remark}

\makeatletter \@addtoreset{equation}{section}

\makeatother \makeatletter

\newcommand{\thmref}[1]{Theorem~\ref{#1}}
\newcommand{\secref}[1]{Section~\ref{#1}}

\newcommand{\lemref}[1]{Lemma~\ref{#1}}
\newcommand{\propref}[1]{Proposition~\ref{#1}}

\newcommand{\eqnref}[1]{(\ref{#1})}

\begin{document}

\title{Central elements  of the degenerate quantum \\ general linear group}
\author{Hengyun Yang}
\address{\it Hengyun Yang: Department of Mathematics, Shanghai Maritime University, Shanghai 201306, China}
\email{hyyang@shmtu.edu.cn}

\author{Yang Zhang}
\address{\it Yang Zhang:  School of Mathematics and Physics, The University of Queensland, St Lucia, QLD 4072, Australia}
\email{yang.zhang@uq.edu.au}

\begin{abstract}
We construct central elements in the degenerate quantum general linear group in the sense of  Cheng, Wang,  and R. B. Zhang \cite{CWZ}, and in particular,  give an explicit formula for the corresponding quantum Casimir element.  Our approach is based on the explicit $L$‑operators, and we further construct a universal $L$-operator, which is a spectral parameter-dependent solution of the quantum Yang-Baxter equation.  This in turn yields an RLL realisation of the degenerate quantum general linear group. Our main results indicate deep links with the quantum general linear supergroup, thereby providing new directions for studying its structure and representations.  
\end{abstract}



\maketitle

\section{Introduction}
Recently, Cheng, Wang and R. B. Zhang \cite{CWZ} introduced a class of new Hopf algebras called
degenerate quantum groups.  This  may be thought of as a degenerate version of the usual Drinfeld-Jimbo quantum groups \cite{D1,J1}. The origin of this idea traces back to the work of Zachos \cite{Za}, where he studied symmetry properties of wave functions of quantum mechanical systems under the action of quantum $\sl_2$ at $\sqrt{-1}$, resulting in a Hopf algebra structure denoted by  ${\rm U}_{q}(\mathfrak{sl}_{1,1})$ (see \eqnref{eq: Zac}). In type $A$,  the degenerate quantum group ${\rm U}_{q}(\gl_{m,n})$ is obtained from the Drinfeld-Jimbo quantum group ${\rm U}_q(\mathfrak{gl}_{m+n})$ by replacing the  subalgebra ${\rm U}_q(\mathfrak{sl}_2)$ associated to the $(m+1)$-th node of the Dynkin diagram of $\mathfrak{gl}_{m+n}$ with Zachos' algebra ${\rm U}_{q}(\mathfrak{sl}_{1,1})$. Additionally, the associated Serre relations are appropriately modified during this process. This construction can be generalised to degenerate  quantum groups of types  $B,C$ and $D$ \cite{CWZ}.

Intriguingly, the classification of  finite dimensional simple modules of ${\rm U}_{q}(\gl_{m,n})$  
 is  the same  \cite[Theorem 4.1]{CWZ} as that for the quantum general linear supergroup ${\rm U}_q(\mathfrak{gl}_{m|n})$ at generic $q$ \cite{Z93}. 
 This  connection is particularly remarkable,  as ${\rm U}_q(\mathfrak{gl}_{m,n})$ is merely 
 an ordinary Hopf algebra containing no odd subspace. Unlike quantum supergroups, 
 ${\rm U}_q(\mathfrak{gl}_{m,n})$ does not arise as a deformation of the universal enveloping algebra 
 of any Lie (super)algebra. Possible connections between ${\rm U}_q(\mathfrak{gl}_{m|n})$ and ${\rm U}_q(\mathfrak{gl}_{m, n})$, 
and between $B$, $C$ and $D$ types of quantum supergroups and degenerate quantum groups 
of the corresponding types \cite[\S 6.1]{CWZ},  were discussed in \cite[\S 6.3]{CWZ}. It indicates that 
degenerate quantum groups may provide a new way to study the theory of quantum supergroups.

In this paper, we are mainly concerned with the centre of the degenerate quantum general linear group ${\rm U}_q(\mathfrak{gl}_{m,n})$.  As a first step,  we construct central elements of ${\rm U}_q(\mathfrak{gl}_{m,n})$ using the method of R. B. Zhang, Gould and Bracken \cite{Z2,ZGB1,ZGB2}. This approach provides a systematic technique for constructing  an infinite family of central elements associated with any finite-dimensional ${\rm U}_q(\mathfrak{gl}_{m,n})$-module. The construction is based on a commutation condition derived from the $L$-operators $L^{\pm}$ of ${\rm U}_q(\mathfrak{gl}_{m,n})$.  We provide explicit constructions of these $L$-operators and, in particular, derive an explicit formula for the quantum Casimir element of ${\rm U}_q(\mathfrak{gl}_{m,n})$. 

Previous works \cite{DZ,L,LXZ} employed a similar method to construct explicit generators and relations for the centre of the Drinfeld-Jimbo quantum (super)group \cite{DZ,L,LXZ,Z2,ZGB1}. It is important to note that  degenerate quantum  groups  are ordinary Hopf algebras rather than quantum deformations of any Lie (super)algebras. A subtle aspect of our construction  is the crucial role played by the element $K_{2\rho}$ (see \eqnref{eq: K2rho}). In the context of quantum groups, this element is typically indexed by the  sum of positive roots; however, here $K_{2\rho}$ appears in a nontrivial manner that ensures consistency with the underlying Hopf algebra structure. As a consequence,   the explicit formula for the quantum Casimir element of $\mathrm{U}_q(\mathfrak{gl}_{m,n})$ (see \eqnref{eq: Cas1} and \eqnref{eq: Cas2}) intriguingly depends on the parity of $m + n$.

Additionally, we construct a one-parameter family of operators  $L(x)\in {\rm U}_q(\mathfrak{gl}_{m,n})\otimes {\rm End}_{\mathbb{C}(q)}(V)$, where $x\in \mathbb{C}^*$ and  $V$ is the natural module of ${\rm U}_q(\mathfrak{gl}_{m,n})$,  and  prove that $L(x)$ satisfies  the quantum Yang-Baxter equation.   These solutions are useful for  constructing new integrable lattice models of quantum systems \cite{SZ}, such as the  celebrated  six-vertex  model and the  related  spin-$\frac{1}{2}$  XXZ   quantum  chain.    In the course of proof, we essentially utilise the RLL relations, which motivate us to formulate the RLL realisation of  ${\rm U}_q(\mathfrak{gl}_{m,n})$. This provides an alternative approach to studying the structure and representation theory  of  ${\rm U}_q(\mathfrak{gl}_{m,n})$.

Spectral parameter $L$-operators and the RLL realisation have long been central to the study of the quantum groups and supergroups \cite{RTF}. It is interesting to see that our construction  exhibits  similarities  to the quantum supergroup case \cite{Z2,Z98}. However,  the  algebraic structure studied here is essentially different.  Consequently, none of our main results follows by analogy alone, and each must be established through careful, explicit analysis.   We anticipate that these similarities will deepen  understanding of quantum supergroup structures and provide new tools for their representation theory.  We will pursue these connections in future work.

The paper is organised as follows. In \secref{sec: pre}, we recall the definition of the degenerate quantum general linear group ${\rm U}_{q}(\gl_{m,n})$ and  prove  some useful commutation relations. In \secref{sec: centre}, we first give a general method to construct  central elements of ${\rm U}_{q}(\gl_{m,n})$, and then use $L$-operators to construct an explicit infinite family of central elements. Our main results are given in Theorem \ref{thm-gen-center} and \thmref{thm-center}. In \secref{sec: Lop}, we construct a spectral parameter-dependent universal $L$-operator. Finally in \secref{sec: RLL}, we present the RLL realisation of the  degenerate quantum general linear group.

\noindent {\bf Acknowledgements.} The authors are grateful to Professor Ruibin Zhang for many stimulating discussions and helpful suggestions. The first author is supported by the National Natural Science Foundation of China (Grant No. 12071276).

\section{The degenerate quantum general linear group}\label{sec: pre}

Let $\C$ be the  field of complex numbers, and  $\Z_+$ be the set of non-negative integers. We set $\mathbb{C}^*:= \mathbb{C}\backslash \{0\}$.  Throughout the paper, we work over $\mathcal{K}=\C (q)$,  the field  of rational functions in the indeterminate $q$.
We fix a pair of positive integers $m,n$. Let ${\bf I}_{m,n}=\{1,2,\dots,m+n\}$ and ${\bf I}_{m,n}'={\bf I}_{m,n}\setminus \{m+n\}$. Put $p=-q^{-1}$, and let $q_a=q$ if $a\le m$, and $q_a=p$ if $a> m$.

\subsection{The degenerate quantum general linear group}
We recall the definition from \cite{CWZ}. The degenerate quantum general linear group
${\rm U}_{q}(\gl_{m,n})$ is  a unital associative algebra over $\mathcal{K}$
generated by the elements $e_a,f_a,K_b,K_b^{-1}, a\in {\bf I}_{m,n}', b\in
{\bf I}_{m,n}$, subject to the following relations
\begin{align}
& K_aK_a^{-1}=K_a^{-1}K_a=1,\quad K_a^{\pm 1}K_b^{\pm1}=K_b^{\pm1}K_a^{\pm 1},\\
& K_ae_bK_a^{-1}=q_a^{\d_{ab}-\d_{a,b+1}}e_b,\\
& K_af_bK_a^{-1}=q_a^{-\d_{ab}+\d_{a,b+1}}f_b, \\
& e_af_b-f_be_a=\d_{ab}\frac{k_{a}-k_{a}^{-1}}{q_{a}-q_{a}^{-1}}, \text{\ with\ } k_a=K_aK_{a+1}^{-1},\\
& e_ae_b=e_be_a,\quad f_af_b=f_bf_a,\quad |a-b|>1,\\
& e_a^2e_{a\pm1}-(q_a+q_a^{-1})e_ae_{a\pm1}e_a+e_{a\pm1}e_a^2=0,\quad a\neq m,\\
& f_a^2f_{a\pm1}-(q_a+q_a^{-1})f_af_{a\pm1}f_a+f_{a\pm1}f_a^2=0,\quad a\neq m,\\
& e_m^2=f_m^2=0, \\
&e_mE_{m-1,m+2}-E_{m-1,m+2}e_m=0, \label{eq: eE} \\
& f_mE_{m+2,m-1}-E_{m+2,m-1}f_m=0, \label{eq: fE}
\end{align}
where $E_{m-1,m+2}$ and $E_{m+2,m-1}$ are defined by
\[
\begin{aligned}
& E_{m-1,m+2}:=E_{m-1,m+1}e_{m+1}-q_{m+1}^{-1}e_{m+1}E_{m-1,m+1},\\
& E_{m+2,m-1}:=f_{m+1}E_{m+1,m-1}-q_{m+1}E_{m+1,m-1}f_{m+1},\\
& E_{m-1,m+1}:=e_{m-1}e_{m+1}-q_{m}^{-1}e_{m}e_{m-1},\\
& E_{m+1,m-1}:=f_{m}f_{m-1}-q_{m}f_{m-1}f_{m}.
\end{aligned}
\]
Denote by ${\rm U}_q(\mathfrak{sl}_{m,n})$ the subalgebra  of ${\rm U}_q(\mathfrak{gl}_{m,n})$ generated by $k_a^{\pm 1}, e_a, f_a$ for all $a\in {\bf I}_{m,n}'$.

Note that if $m=1$ or $n=1$, relations \eqref{eq: eE} and \eqref{eq: fE} do not exist. For any $a\neq m$, the elements $e_a, f_a, k_a^{\pm 1}$ generate a copy of ${\rm U}_q(\mathfrak{sl}_2)$. If $a=m$, the elements $k_{m}^{\pm}, e_m, f_m$ generate a subalgebra isomorphic to ${\rm U}_q(\mathfrak{sl}_{1,1})$  with relations
\begin{equation}\label{eq: Zac}
\begin{aligned}
&k_mk_m^{-1}=k_m^{-1}k_m=1, \quad k_m e_mk_m^{-1}= -e_m, \quad k_m f_mk_m^{-1}= -f_m, \\
&e_mf_m-f_me_m= \frac{k_m-k_m^{-1}}{q-q^{-1}}, \quad e_m^2=f_m^2=0.
\end{aligned}
\end{equation}
This is Zachos' algebra \cite{Za}, whose defining relations are similar to these of the quantum supergroup ${\rm U}_q(\mathfrak{sl}_{1|1})$ but without $\mathbb{Z}_2$-grading.  However, despite the resemblance,  Zachos' algebra is not a quantum deformation of any Lie algebra or Lie superalgebra \cite{CWZ,Za}.

As with the usual Drinfeld-Jimbo quantum group, the degenerate quantum group ${\rm U}_q(\mathfrak{gl}_{m,n})$ has a Hopf algebra structure with a coproduct $\D:{\rm U}_{q}(\gl_{m,n})\to {\rm U}_{q}(\gl_{m,n})\otimes {\rm U}_{q}(\gl_{m,n})$,
a counit $\epsilon:{\rm U}_{q}(\gl_{m,n})\to \mathcal{K}$, and an antipode $S:{\rm U}_{q}(\gl_{m,n})\to {\rm U}_{q}(\gl_{m,n})$, which are defined, respectively, by
\begin{equation}\label{def-bialg}
\D(e_a)=e_a\otimes k_a+1\otimes e_a,\quad \D(f_a)=f_a\otimes 1+ k_a^{-1}\otimes f_a,\quad \D(K_b)=K_b\otimes K_b,
\end{equation}
$$\epsilon(e_a)=\epsilon(f_a)=0,\quad \epsilon(K_b)=1,$$
and
\begin{equation}\label{eq: antipode}
S(e_a)=-e_ak_a^{-1},\quad S(f_a)=-k_a f_a,\quad S(K_b)=K_b^{-1}
\end{equation}
for all $a\in {\bf I}_{m,n}', b\in {\bf I}_{m,n}$. We define the  opposite coproduct $ \Delta'$ by
\[
\Delta':= \sigma \Delta,
\]
where $\sigma(u_1\otimes u_2)= u_2\otimes u_1$ for any $u_1,u_2 \in {\rm U}_q(\mathfrak{gl}_{m,n})$.

Let $V= \mathcal{K}^{m+n}$, and let $v_a\in V$ ($a\in {\bf I}_{m,n}$) be the column vector with $1$ at the $a$-th entry and $0$ elsewhere. Let $e_{ab}, a,b\in {\bf I}_{m,n}, $ be the matrix units such that $e_{ab} v_c= \delta_{bc}v_a$ for all $a,b,c\in {\bf I}_{m,n}$. Then there is a natural representation $\pi:{\rm U}_{q}(\gl_{m,n})\to \End_{\mathcal{K}}(V)$ \cite[Lemma 4.3]{CWZ}, defined by
\begin{equation}\label{eq: nat}
\pi(e_a)=e_{a,a+1},\quad \pi(f_a)=e_{a+1,a},\quad \pi(K_b)=I+(q_b-1)e_{bb}, \quad a\in {\bf I}_{m,n}', b\in {\bf I}_{m,n},
\end{equation}
where $I=\sum_{a\in \mathbf{I}_{m,n}}e_{aa}$ denotes the identity matrix. 
For any $u\in {\rm U}_{q}(\gl_{m,n})$, we have $$uv_a=\sum_{b\in {\bf I}_{m,n}}\pi(u)_{ba}v_b,\quad a\in {\bf I}_{m,n}.$$ 

The classification of  finite dimensional simple ${\rm U}_q(\mathfrak{gl}_{m,n})$-modules is  essentially the same as that for the quantum general linear supergroup ${\rm U}_q(\mathfrak{gl}_{m|n})$ (or $\mathfrak{gl}_{m|n}$) \cite[Theorem 4.1]{CWZ}.

\subsection{Commutation relations}

For any $a,b\in {\bf I}_{m,n}$, let
\[
E_{a,a+1}={\bar E}_{a,a+1}=e_a, \quad  E_{a+1,a}={\bar E}_{a+1,a}=f_a.
\]
We define recursively the following elements of
${\rm U}_{q}(\gl_{m,n})$:
\begin{equation}\label{def-E}
E_{ab}=\begin{cases}
E_{ac}E_{cb}-q_c^{-1}E_{cb}E_{ac}, & a<c<b,\\
E_{ac}E_{cb}-q_cE_{cb}E_{ac}, & b<c<a, \end{cases}
\end{equation}
and
\begin{equation}\label{def-E'}
\bar{E}_{ab}=\begin{cases}
\bar{E}_{ac}\bar{E}_{cb}-q_c\bar{E}_{cb}\bar{E}_{ac}, & a<c<b,\\
\bar{E}_{ac}\bar{E}_{cb}-q_c^{-1}\bar{E}_{cb}\bar{E}_{ac}, & b<c<a. \end{cases}
\end{equation}
We will give some commutation relations among these elements. These relations will be used in the construction of central elements of ${\rm U}_q(\mathfrak{gl}_{m,n})$.

\begin{lem}\label{lem: KErel}
	For any $a\neq b\in {\bf I}_{m,n}$,  the following relations hold in  ${\rm U}_{q}(\gl_{m,n})$:
	\begin{align}
	&K_cE_{ab}=E_{ab}K_c, \quad  c\neq a,b, \label{eq:KE1}\\
	&K_aE_{ab}=q_aE_{ab}K_a,   \quad
	K_bE_{ab}=q_b^{-1}E_{ab}K_b.   \label{eq:KE3}
	\end{align}
	These relations similarly apply to $\bar{E}_{ab}$ mutatis mutandis.
\end{lem}
\begin{proof}
	We prove the first relation. We may assume that $a<b$, and the case $a>b$ can be treated similarly. By definition (\ref{def-E})  $E_{ab}$ (resp.
	$E_{ba}$) is a linear combination of elements of the form
	$e_{i_1}e_{i_{2}}\cdots e_{i_{b-a}}$ (resp.
	$f_{i_1}f_{i_{2}}\cdots f_{i_{b-a}}$), where $i_1, i_2, \dots, i_{b-a}$ form a permutation of $a,a+1,\dots,b-1$. Assuming that $a<c<b$, we obtain
	$K_ce_{c-1}=q_c^{-1}e_{c-1}K_c$, $K_ce_{c}=q_ce_{c}K_c$
	and $K_ce_{i}=e_{i}K_c$ for $i\neq c, c-1$. Then $K_cE_{ab}=E_{ab}K_c$ for $a<c<b$. Similarly,  (\ref{eq:KE1})
	holds for the cases $a<b<c$ and $c<a<b$. Relations in  (\ref{eq:KE3}) can be proved similarly.
\end{proof}

\begin{lem}\label{eE}
	We have the following relations:
	\begin{align}
	&[e_{a},E_{ba}]=-k_{a}E_{b,a+1}, \quad a+1<b, \label{eq:11}\\
	& [f_{a},E_{ab}]=E_{a+1,b}k_{a}^{-1},\quad a+1<b, \label{eq: 111}  \\
	&[e_{b},E_{b+1,a}]=E_{ba}k_{b}^{-1}, \quad a<b, \label{eq: 112} \\
	& [f_{b},E_{a,b+1}]=-k_{b} E_{ab}, \quad a<b, \label{eq:12}\\
	& [e_a,E_{bc}]=[f_a,E_{cb}]=0,\quad b<c, \{b,c\}\neq \{a,a+1\}. \label{eq:13}
	\end{align}
\end{lem}
\begin{proof}
	Let us prove (\ref{eq:11}) first. For any $a+1<b$, we have
	\begin{align*}
	[e_{a},E_{ba}]&=[e_{a}, E_{b,a+1}E_{a+1,a}-q_{a+1}E_{a+1,a}E_{b,a+1}]\\
	&=[e_a, E_{b,a+1}]f_a + E_{b,a+1}[e_a,f_a] - q_{a+1}([e_a,f_a]E_{b,a+1}+ f_a[e_a,E_{b,a+1}]).
	\end{align*}
	We can prove by induction  that $[e_a,E_{b,a+1}]=0$. It follows that
	\begin{align*}
	[e_{a},E_{ba}]&=E_{b,a+1}\frac{k_{a}-k_{a}^{-1}}{q_{a}-q_{a}^{-1}}-q_{a+1}\frac{k_{a}-k_{a}^{-1}}{q_{a}-q_{a}^{-1}}E_{b,a+1}=-k_{a}E_{b,a+1},
	\end{align*}
	where in the last equation we have used \lemref{lem: KErel}.
	Relations \eqref{eq: 111}, \eqref{eq: 112} and \eqref{eq:12} can be proved similarly. The proof of $[f_a,E_{cb}]=0$  in \eqref{eq:13} can be found in Lemma 3.8 \cite{CWZ}, and the first relation in \eqref{eq:13} can be proved similarly.
\end{proof}

\begin{lem} \label{lem: Erel}
	The following relations hold:
	\begin{eqnarray}
	&&E_{ca}^2=0,\quad  a\leq m<c, \label{eq:E10}\\
	&&{[E_{ab},E_{cd}]}=0, \quad  b<a<d<c \text{ or }  d<b<a<c, \label{eq:E11} \\
	&&E_{ac}E_{bc}=q_cE_{bc}E_{ac} , \quad E_{ca}E_{cb}=q_cE_{cb}E_{ca},\quad c<a<b,\label{eq:E13} \\
	&&E_{ca}E_{cb}=q_cE_{cb}E_{ca},\quad E_{ac}E_{bc}=q_cE_{bc}E_{ac} , \quad a<b<c,\label{eq:E15} \\
	&&{[E_{ab},E_{cd}]}=(q-q^{-1})E_{cb}E_{ad}, \quad  b<d<a<c \text{ or } a<c<b<d, \label{eq:E14}\\
	&&[E_{ab},E_{bc}]_{q_b^{-1}}=E_{ac}, \quad [E_{cb},E_{ba}]_{q_b}=E_{ca}\quad a<b<c. \label{eq:E12}
	\end{eqnarray}
\end{lem}
\begin{proof}
	The proof of relations \eqref{eq:E10}, \eqref{eq:E11} and the first relations of \eqref{eq:E13} and \eqref{eq:E15}, \eqref{eq:E14} can be found in \cite[Lemma 3.9]{CWZ}.	The second relations of \eqref{eq:E13} and \eqref{eq:E15} can be proved by using similar method as the first relations of \eqref{eq:E13} and \eqref{eq:E15}, respectively. From the fact that $e_ae_c=e_ce_a$ and $f_af_c=f_cf_a$ for $|a-c|>1$, we can easily prove \eqref{eq:E12}.
\end{proof}

\begin{lem}\label{S-act on E} For $a>b$, we have
	$$S(E_{ab})=-K_a^{-1}K_b\bar{E}_{ab}, \quad S(E_{ba})= -\bar{E}_{ba}K_aK_{b}^{-1}.$$
\end{lem}
\begin{proof}
	We prove the first equation by using induction on $a-b$, and similarly for the second. If $a-b=1$, by \eqref{eq: antipode} we have
	\[
	S(E_{a,a-1})=-k_{a-1}f_{a-1}=-K_a^{-1}K_{a-1}\bar{E}_{a,a-1}.
	\]
	Assume the formula  for $S(E_{a-1,b})$ holds for $a-1>b$. Using definition \eqref{def-E} we obtain that
	\[
	\begin{aligned}
	& S(E_{ab})=S(E_{a-1,b})S(E_{a,a-1})-q_{a-1}S(E_{a,a-1})S(E_{a-1,b})\\
	={}& K_{a-1}^{-1}K_b\bar{E}_{a-1,b}K_a^{-1}K_{a-1}\bar{E}_{a,a-1}-q_{a-1}K_a^{-1}K_{a-1}\bar{E}_{a,a-1}K_{a-1}^{-1}K_b\bar{E}_{a-1,b}\\
	={}& q_{a-1}^{-1}K_a^{-1}K_b\bar{E}_{a-1,b}\bar{E}_{a,a-1}-K_a^{-1}K_b\bar{E}_{a,a-1}\bar{E}_{a-1,b}\\
	={}& -K_a^{-1}K_b\bar{E}_{ab},
	\end{aligned}
	\]	
	where the second equation follows from the induction hypothesis and the third equation is a consequence of \lemref{lem: KErel}.
\end{proof}

\section{Central elements of \texorpdfstring{${\rm U}_{q}(\gl_{m,n})$}{Uq(glmn)}}\label{sec: centre}

We will begin by using the partial trace technique to construct central elements of ${\rm U}_q(\mathfrak{gl}_{m,n})$ associated with an arbitrary finite-dimensional ${\rm U}_q(\mathfrak{gl}_{m,n})$-module. Subsequently, we provide an explicit construction of central elements associated with the natural module. In particular, this yields an explicit quantum Casimir element of ${\rm U}_q(\mathfrak{gl}_{m,n})$. 

\subsection{A general construction}
Although ${\rm U}_q(\mathfrak{gl}_{m,n})$ is not a quantum deformation of any existing Lie algebra or Lie superalgebra, we will show in this section that the method for constructing central elements for quantum (super) groups, as given in \cite{DZ,ZGB1}, is applicable to ${\rm U}_q(\mathfrak{gl}_{m,n})$. This is made possible due to its Hopf algebra structure.

To start with, let $M$ be an arbitrary finite dimensional ${\rm U}_{q}(\gl_{m,n})$-module, and let $\zeta:{\rm U}_{q}(\gl_{m,n})\to {\rm End}_{\mathcal{K}}(M)$ be the corresponding representation of ${\rm U}_{q}(\gl_{m,n})$. Denote by ${\rm Tr}_2$ the partial trace on the second tensor factor of ${\rm U}_q(\mathfrak{gl}_{m,n})\otimes {\rm End}_{\mathcal{K}}(M)$, i.e.,
\[ {\rm Tr}_2(u\otimes A)={\rm Tr}(A)u, \quad \forall u\in {\rm U}_q(\mathfrak{gl}_{m,n}), A\in {\rm End}_{\mathcal{K}}(M), \]
where ${\rm Tr}$ denotes the usual trace.

An important ingredient in the construction is the  element $K_{2\rho}$.  In the context of  quantum  groups, $2\rho$ denotes the sum of positive roots and $K_{2\rho}$ is merely a product of  $K_{\alpha}$ over all positive roots $\alpha$. However, when dealing with  ${\rm U}_q(\mathfrak{gl}_{m,n})$, it is crucial to define $K_{2\rho}$ properly to ensure consistency with the antipode $S$.    Recall from \cite[Lemma 5.1]{CWZ} that $K_{2\rho}\in {\rm U}_q(\mathfrak{gl}_{m,n})$ is defined by
\begin{equation}\label{eq: K2rho}
K_{2\rho}=\begin{cases}
K_{2\rho}', \quad & \text{if $m+n$ is even},\\
K_{2\rho}'K', \quad &\text{if $m+n$ is odd},
\end{cases}
\end{equation}
where
\[
K_{2\rho}'=\prod_{a=1}^m K_a^{m-n+1-2a}\prod_{b=1}^n K_{m+b}^{m+n+1-2b}, \quad K'= \prod_{a=1}^mK_a \prod_{b=1}^n K^{-1}_{m+b}.
\]
This invertible element satisfies
\[
S^2(u)= K_{2\rho} u K_{2\rho}^{-1}, \quad \forall u\in {\rm U}_q(\mathfrak{gl}_{m,n}),
\]
which is analogous to the usual quantum group case. In particular, it is straightforward to check the following useful relations:
\begin{equation}\label{eq: Ke}
\begin{aligned}
K_{2\rho}'e_aK_{2\rho}^{'-1}&=k_ae_ak_a^{-1}=q_a^2e_a, \quad \forall a\neq m, \\
K'e_aK^{'-1}&=e_a, \quad \forall a\neq m,\\
K_{2\rho}e_m K_{2\rho}^{-1}&=k_me_mk_m^{-1}=-e_m.
\end{aligned}
\end{equation}

Now we apply the partial trace technique (see, e.g., \cite[Proposition 1]{ZGB1}) to construct  central elements of ${\rm U}_q(\mathfrak{gl}_{m,n})$.

\begin{thm}\label{thm-gen-center}
	Let $\zeta: {\rm U}_q(\mathfrak{gl}_{m,n}) \rightarrow {\rm End}_{\mathcal{K}}(M)$ be a finite-dimensional representation, and let $\Gamma_M\in  {\rm U}_q(\mathfrak{gl}_{m,n})\otimes {\rm End}_{\mathcal{K}}(M)$ be an element satisfying
	\[[\Gamma_M, (1\otimes \zeta)\Delta(u)]=0, \quad \forall u\in {\rm U}_q(\mathfrak{gl}_{m,n}).\]
	Then the elements
	\[ C_{k}= {\rm Tr}_2((1\otimes \zeta(K_{2\rho}))\Gamma_M^k), \quad k\geq 1 \]
	belong to the centre of ${\rm U}_q(\mathfrak{gl}_{m,n})$.
\end{thm}
\begin{proof}
	For any integer $k\geq 1$, it suffices to show that $C_k$ commutes with the generators $e_a,f_a, K_b^{\pm 1}$, $a\in {\bf I}_{m,n}', b\in {\bf I}_{m,n}$. Suppose that $\Gamma_M^k= \sum_{i}x_i\otimes X_i$ is a finite sum for some $x_i\in  {\rm U}_q(\mathfrak{gl}_{m,n})$ and $X_i\in {\rm End}_{\mathcal{K}}(M)$.  Then for any $a\in {\bf I}_{m,n}'$,  we have $[\Gamma_M^k,(1\otimes \zeta)\D(e_a)]=0$.  It follows that
	\[\begin{aligned}
	0&= {\rm Tr}_2\Big ( (1\otimes \zeta(k_a^{-1}K_{2\rho}))[\Gamma_M^k,(1 \otimes \zeta)\Delta(e_a) ]\Big)\\
	&=\sum_i {\rm Tr}_2\Big ( (1\otimes \zeta(k_a^{-1}K_{2\rho})) [x_i\otimes X_i, e_a\otimes \zeta(k_a)+1\otimes \zeta(e_a) ]\Big) \\
	&=\sum_i (x_ie_a-e_ax_i){\rm Tr}(\zeta(K_{2\rho})X_i) + \sum_ix_i\Big( {\rm Tr}(\zeta(k_a^{-1}K_{2\rho})X_i\zeta(e_a )) - {\rm Tr}( \zeta(k_a^{-1}K_{2\rho}e_a)X_i) \Big)\\
	&= [C_k, e_a]+ \sum_ix_i\Big( {\rm Tr}(\zeta(e_ak_a^{-1}K_{2\rho})X_i) - {\rm Tr}( \zeta(k_a^{-1}K_{2\rho}e_a)X_i) \Big).
	\end{aligned}\]
	We will show that the sum in the last equation equals zero. This implies that $[C_k,e_a]=0$, proving that $C_k$ commutes with $e_a$.
	Note that for each term in the sum, we have
	\[
	{\rm Tr}(\zeta(e_ak_a^{-1}K_{2\rho})X_i)=
	\begin{cases}
	-{\rm Tr}(\zeta(k_m^{-1}e_mK_{2\rho})X_i), \quad & \text{if $a=m$}, \\
	q_a^2 {\rm Tr}(\zeta(k_a^{-1}e_aK_{2\rho})X_i), \quad & \text{if $a\neq m$}.
	\end{cases}
	\]
	If $a=m$, by the last equation of \eqref{eq: Ke} we have $K_{2\rho}e_m=-e_m K_{2\rho}$, whence for each $i$, 
	\[
	{\rm Tr}(\zeta(e_mk_m^{-1}K_{2\rho})X_i)= -{\rm Tr}(\zeta(k_m^{-1}e_mK_{2\rho})X_i)= {\rm Tr}(\zeta(k_m^{-1}K_{2\rho}e_m)X_i).
	\]
	Therefore, we have  $[C_k,e_m]=0$. If $a\neq m$, then
	\[
	e_aK_{2\rho}= \begin{cases}
	e_aK_{2\rho}'= q_a^{-2}K'_{2\rho}e_a=q_a^{-2}K_{2\rho}e_a , \quad & \text{if $m+n$ is even},\\
	e_aK_{2\rho}'K'= q_a^{-2}K'_{2\rho}e_aK'=q_a^{-2}K_{2\rho}e_a , \quad &\text{if $m+n$ is odd},
	\end{cases}
	\]
	where we have used relations in \eqref{eq: Ke}. It follows that for each $i$,
	\[
	{\rm Tr}(\zeta(e_ak_a^{-1}K_{2\rho})X_i)=q_a^2 {\rm Tr}(\zeta(k_a^{-1}e_aK_{2\rho})X_i)= {\rm Tr}(\zeta(k_a^{-1}K_{2\rho}e_a)X_i).
	\]
	Thus, we obtain $[C_k,e_a]=0$ for $a\neq m$. Similarly, one can prove that $[C_k,f_a]= [C_k, K_b]=0$ for $1\leq a\leq m+n-1$ and $1\leq b\leq m+n$. Therefore, the elements $C_k$ are central in ${\rm U}_q(\mathfrak{gl}_{m,n})$.
\end{proof}

\subsection{Explicit formulae for central elements}
We will construct an explicit element $\Gamma_M$ which satisfies the condition in \thmref{thm-gen-center}. In this way, we obtain corresponding central elements $C_k$ for all $k\geq 1$.  In the following, we  are concerned with $M=V=\mathcal{K}^{m+n}$, the natural representation of ${\rm U}_{q}(\gl_{m,n})$ as given in \eqref{eq: nat}.

We define the following two elements of ${\rm U}_{q}(\gl_{m,n})\otimes \End_{\mathcal{K}} (V)$:
\begin{eqnarray}
&&L^{+}=\sum\limits_{a\in {\bf I}_{m,n}}K_a\otimes e_{aa}+(q-q^{-1})\sum\limits_{a<b}K_bE_{ab}\otimes e_{ba},\label{def-L+}\\
&&L^{-}=\sum\limits_{a\in {\bf I}_{m,n}}K_a^{-1}\otimes e_{aa}-(q-q^{-1})\sum\limits_{a<b}E_{ba} K_b^{-1}\otimes e_{ab}.\label{def-Lmin}
\end{eqnarray}

The definition of $L^{\pm}$  is inspired by the  connection between ${\rm U}_q(\mathfrak{gl}_{m,n})$ and ${\rm U}_q(\mathfrak{gl}_{m|n})$ discussed in the introduction, as well as the explicit $L$-operators for ${\rm U}_q(\mathfrak{gl}_{m|n})$ given in \cite{Z2}. However, it is not immediately apparent from this connection that the operators $L^{\pm}$ satisfy the properties stated in \lemref{lem: SL} and \lemref{lem-L-D} below, as ${\rm U}_q(\mathfrak{gl}_{m,n})$ and ${\rm U}_q(\mathfrak{gl}_{m|n})$  have fundamentally different algebraic structures.  Therefore, a detailed analysis is necessary and, in fact, turns out to be nontrivial.

The element $L^{-}$ has  the inverse given as follows.
\begin{lem}\label{lem: SL}
	Recall that $S$ is the antipode of ${\rm U}_{q}(\gl_{m,n})$ given by \eqref{eq: antipode}. Then
	\begin{equation}\label{def-L--}
	(L^{-})^{-1}=(S\otimes 1)(L^{-})=
	\sum\limits_{a\in {\bf I}_{m,n}}K_a\otimes e_{aa}+(q-q^{-1})\sum\limits_{a<b}K_a\bar{E}_{ba}\otimes e_{ab}.
	\end{equation}
	
\end{lem}

\begin{proof}	
	Using Lemma \ref{S-act on E}, we obtain
	\[
	\begin{aligned}
	(S\otimes 1)(L^{-})
	={}&\sum\limits_{a\in {\bf I}_{m,n}}K_a\otimes e_{aa}-(q-q^{-1})\sum\limits_{a<b}S(E_{ba}K_b^{-1})\otimes e_{ab}\\
	={}&\sum\limits_{a\in {\bf I}_{m,n}}K_a\otimes e_{aa}+(q-q^{-1})\sum\limits_{a<b}K_a \bar{E}_{ba}\otimes e_{ab}.
	\end{aligned}
	\]
	This proves the second equation. Next we check directly that $(S\otimes 1)(L^{-})$ is the inverse of $L^{-}$. Note that
	\[
	\begin{aligned}
	L^{-}((S\otimes 1)(L^{-}))=
	1\otimes I +(q-q^{-1})\sum\limits_{a<b}(\bar{E}_{ba}-E_{ba}-(q-q^{-1})\sum\limits_{c=a+1}^{b-1}E_{ca}\bar{E}_{bc})\otimes e_{ab}.
	\end{aligned}
	\]
	We claim that
	\begin{equation*}
	\bar{E}_{ba}=E_{ba}+(q-q^{-1})\sum\limits_{c=a+1}^{b-1}E_{ca}\bar{E}_{bc},\quad a<b,
	\end{equation*}
	and hence $L^{-}((S\otimes 1)(L^{-}))=1\otimes I$.  To prove this claim, we use induction on $b-a$.  If $b=a+1$, then $\bar{E}_{ba}=E_{ba}=f_a$ holds true by definition.  Assume that the claim holds for $\bar{E}_{b-1,a}$ for $b-a>1$.  Then we obtain
	\[
	\begin{aligned}
	&\bar{E}_{ba}=\bar{E}_{b,b-1}\bar{E}_{b-1,a}-q_{b-1}^{-1}\bar{E}_{b-1,a}\bar{E}_{b,b-1}\\
	=& E_{b,b-1}(E_{b-1,a}+(q-q^{-1})\sum\limits_{c=a+1}^{b-2}E_{ca}\bar{E}_{b-1,c})-q_{b-1}^{-1}(E_{b-1,a}+(q-q^{-1})\sum\limits_{c=a+1}^{b-2}E_{ca}\bar{E}_{b-1,c})E_{b,b-1}\\
	=& E_{b,b-1}E_{b-1,a}-q_{b-1}^{-1}E_{b,b-1}E_{b-1,a} +(q-q^{-1})\sum\limits_{c=a+1}^{b-2}E_{ca}(E_{b,b-1}\bar{E}_{b-1,c}-q_{b-1}^{-1}\bar{E}_{b-1,c}E_{b,b-1})\\
	=& E_{ba}+(q_{b-1}-q_{b-1}^{-1})E_{b-1,a}E_{b,b-1}+(q-q^{-1})\sum\limits_{c=a+1}^{b-2}E_{ca}\bar{E}_{bc}   \\
	=&E_{ba}+(q-q^{-1})\sum\limits_{c=a+1}^{b-1}E_{ca}\bar{E}_{bc}.
	\end{aligned}
	\]
	Therefore, the claim holds. Similarly, one can prove that $((S\otimes 1)(L^{-}))L^{-}=1\otimes I$, and hence $(S\otimes 1)(L^{-})$ is the inverse of $L^{-}$.
\end{proof}

Recall that $\pi$ is the natural representation of $ {\rm U}_{q}(\gl_{m,n})$.
The elements $L^{\pm}$ have the following important property, which is  analogous to their counterparts in  the case of ${\rm U}_{q}(\mathfrak{gl}_{m|n})$ \cite{Z2}.

\begin{lem}\label{lem-L-D}
	For any $u\in {\rm U}_{q}(\gl_{m,n})$, we have
	\begin{equation}\label{prop-L+-}
	L^{\pm}(1 \otimes \pi)(\D(u))=(1 \otimes \pi)(\D'(u))L^{\pm}.
	\end{equation}
\end{lem}
\begin{proof}
	We introduce the following elements:
	\begin{align*}
	\widetilde{L}^{+}:&=1\otimes I+(q-q^{-1})\sum\limits_{a<b}E_{ab}\otimes e_{ba},\\
	\widetilde{L}^{-}:&=1\otimes I-(q-q^{-1})\sum\limits_{a<b}E_{ba} \otimes e_{ab},
	\end{align*}
	which are related to $L^{\pm}$ by
	\begin{eqnarray*}
		L^{+}=(\sum\limits_{a\in {\bf I}_{m,n}}K_a\otimes e_{aa})\widetilde{L}^{+},\quad L^{-}=\widetilde{L}^{-}(\sum\limits_{a\in {\bf I}_{m,n}}K_a^{-1}\otimes e_{aa}).\label{def-L+-}
	\end{eqnarray*}
	To prove \eqref{prop-L+-}, it suffices to  show the following equations:
	\begin{align}
	\widetilde{L}^+( E_{c,c+1}\otimes \pi(k_c)+  1\otimes e_{c,c+1}) &=(E_{c,c+1}\otimes \pi(k_c^{-1})+1\otimes e_{c,c+1}) \widetilde{L}^+, \label{eq: Lplus1} \\
	\widetilde{L}^+( E_{c+1,c}\otimes I + k_{c}^{-1}\otimes e_{c+1,c} ) &= (E_{c+1,c}\otimes I+ k_c \otimes e_{c+1,c}) \widetilde{L}^+, \label{eq: Lplus2} \\
	\widetilde{L}^-(E_{c,c+1} \otimes I+ k_c^{-1} \otimes e_{c,c+1} ) &=(E_{c,c+1}\otimes I +k_c \otimes e_{c,c+1} )\widetilde{L}^{-}, \label{eq: Lmin1} \\
	\widetilde{L}^- (E_{c+1,c}\otimes \pi(k_c) +1\otimes e_{c+1,c})&= (E_{c+1,c}\otimes \pi(k_c^{-1})+ 1\otimes e_{c+1,c}) \widetilde{L}^{-}. \label{eq: Lmin2}
	\end{align}

	We will only prove \eqnref{eq: Lplus1} and \eqnref{eq: Lplus2}, and the other two can be proved similarly. Consider \eqnref{eq: Lplus1}, and define
	\[
	\begin{aligned}
	A_1&= \widetilde{L}^+(E_{c,c+1}\otimes \pi(k_c))- (E_{c,c+1}\otimes \pi(k_c^{-1})) \widetilde{L}^+,\\
	A_2&=[\widetilde{L}^+,  (1\otimes e_{c,c+1})].
	\end{aligned}
	\]
	Then \eqnref{eq: Lplus1} is equivalent to $A_1+A_2=0$. Note that
	\[
	\begin{aligned}
	A_1 = E_{c,c+1}\otimes (\pi(k_c)- \pi(k_c^{-1})) +(q-q^{-1}) \sum_{a<b}A_{ab},
	\end{aligned}
	\]
	where
	\[A_{ab}= E_{ab}E_{c,c+1}\otimes e_{ba}\pi(k_c)-  E_{c,c+1}E_{ab} \otimes \pi(k_c^{-1})e_{ba}. \]
	By \eqnref{eq: nat}, we have
	\[
	\begin{aligned}
	e_{ba}\pi(k_c) &= e_{ba}+(q_c-1)\delta_{ac}e_{bc} + (q_{c+1}^{-1}-1) \delta_{a,c+1} e_{b,c+1},\\
	\pi(k_c^{-1})e_{ba}&= e_{ba}+ (q_c^{-1}-1)\delta_{bc}e_{ca}+(q_{c+1}-1)\delta_{b,c+1}e_{c+1,a}.
	\end{aligned}
	\]
	It follows that
	$$\sum\limits_{a<b}A_{ab}=\sum\limits_{a<c,b=c,c+1}A_{ab}+\sum\limits_{b>c+1,a=c,c+1}A_{ab}+\sum\limits_{a=c,b=c+1}A_{ab}.$$
	Applying \eqnref{def-E} and \lemref{lem: Erel}, we obtain
	\[
	\begin{aligned}
	\sum_{a<c,b=c,c+1}A_{ab}= & \sum_{a<c}(E_{ac}E_{c,c+1}\otimes e_{ca}- E_{c,c+1}E_{ac}\otimes q_{c}^{-1} e_{ca} ) \\
	&+ \sum_{a<c} (E_{a,c+1} E_{c,c+1} \otimes e_{c+1,a} - E_{c,c+1} E_{a, c+1} \otimes q_{c+1} e_{c+1,a})\\
	=& \sum_{a<c} E_{a,c+1} \otimes e_{ca}.
	\end{aligned}
	\]
	Similarly, one can deduce that
	\[
	\begin{aligned}
	\sum\limits_{b>c+1,a=c,c+1}A_{ab}&= -\sum_{b>c+1} E_{cb} \otimes e_{b,c+1},\\
	\sum_{a=c, b=c+1}A_{ab}& =(q_c-q_{c+1})E_{c,c+1}^2 \otimes  e_{c+1,c}=0.
	\end{aligned}
	\]
	Therefore, we have
	\[
	\begin{aligned}
	A_1= & E_{c,c+1}\otimes (\pi(k_c)- \pi(k_c^{-1})) +(q-q^{-1}) \sum_{a<b}A_{ab} \\
	=& (q-q^{-1}) E_{c,c+1} \otimes (e_{cc}-e_{c+1,c+1}) +(q-q^{-1}) \big(\sum_{a<c} E_{a,c+1} \otimes e_{ca} -\sum_{b>c+1} E_{cb} \otimes e_{b,c+1}\big).
	\end{aligned}
	\]
	On the other hand, we have
	\[
	A_2= [\widetilde{L}^+,  (1\otimes e_{c,c+1})]= (q-q^{-1}) (\sum_{b>c} E_{cb}\otimes e_{b,c+1} -\sum_{a<c}E_{a,c+1}\otimes e_{ca} ).
	\]
	It is clear that $A_1+A_2=0$, proving \eqnref{eq: Lplus1}.
	
	We proceed to prove \eqnref{eq: Lplus2}. Define
	\[
	B_1=[\widetilde{L}^+, E_{c+1,c}\otimes I]= -(q-q^{-1})\sum_{a<b} [E_{c+1,c}, E_{ab} ]\otimes e_{ba}.
	\]
	Applying \lemref{eE}, we obtain
	\[
	\begin{aligned}
	B_1&= -(q-q^{-1}) \big([E_{c+1,c},E_{c,c+1}]\otimes e_{c+1,c} +\sum_{a<c}[E_{c+1,c}, E_{a,c+1}]\otimes e_{c+1,a} + \sum_{b>c+1}[E_{c+1,c},,E_{cb}]\otimes e_{bc}  \big)\\
	&= -(q-q^{-1}) \big( -\frac{k_c-k_c^{-1}}{q_c-q_c^{-1}}\otimes e_{c+1,c} +\sum_{a<c} -k_c E_{ac}\otimes e_{c+1, a} +\sum_{b>c+1} E_{c+1,b}k_c^{-1} \otimes e_{bc} \big)\\
	&= (k_c-k_c^{-1})\otimes e_{c+1,c} + (q-q^{-1}) \big( \sum_{a<c} k_c E_{ac}\otimes e_{c+1, a} -\sum_{b>c+1} E_{c+1,b}k_c^{-1} \otimes e_{bc}   \big).
	\end{aligned}
	\]
	On the other hand, we have
	\[
	\begin{aligned}
	B_2&= \widetilde{L}^+ (k_c^{-1}\otimes e_{c+1,c}) -(k_c\otimes e_{c+1,c}) \widetilde{L}^+\\
	&= (k_c^{-1}-k_c)\otimes e_{c+1,c} + (q-q^{-1}) \big( -\sum_{a<c} k_c E_{ac}\otimes e_{c+1, a} +\sum_{b>c+1} E_{c+1,b}k_c^{-1} \otimes e_{bc} \big).
	\end{aligned}
	\]
	Clearly, $B_1+B_2=0$, which is equivalent to \eqnref{eq: Lplus2}. This finishes the proof of the lemma.
\end{proof}

Now we define the following element $\Gamma_V\in {\rm U}_{q}(\gl_{m,n})\otimes \End_{\mathcal{K}} (V)$ associated to the natural representation $V$:
\begin{equation}\label{eq: gammaV}
\Gamma_V:=(L^{-})^{-1}L^{+}.
\end{equation}

{}
\begin{lem}\label{lem: GammaV}
	For any $u\in {\rm U}_{q}(\gl_{m,n})$, we have
	$$[\Gamma_V,( 1 \otimes \pi)(\D(u))]=0.$$
\end{lem}
\begin{proof}
	Using \lemref{lem-L-D}, for any $u\in \mathrm{U}_q(\mathfrak{gl}_{m,n})$,  we have
	\[
	(L^{-})^{-1}L^{+}( 1 \otimes \pi)(\D(u))=  (L^{-})^{-1}( 1 \otimes \pi) (\D'(u)) L^+= ( 1 \otimes \pi)(\D(u)) (L^-)^{-1}L^+.
	\]
	This completes the proof.
\end{proof}

Combining \lemref{lem: GammaV} and \thmref{thm-gen-center}, we obtain explicit central elements $C_k$ of ${\rm U}_q(\mathfrak{gl}_{m,n})$ associated to $V$ for any $k\geq 1$.

\begin{example}\label{exam: simp}
	Let ${\rm U}_{q}(\mathfrak{gl}_{1,1})$ be the degenerate quantum group generated by $K_{1}^{\pm 1}, K_2^{\pm 1}, E=E_{12}$ and $F=E_{21}$, subject to the following relations:
	\begin{align*}
	& K_1^{\pm 1} K_2^{\pm 1}=K_2^{\pm 1}K_1^{\pm 1}, \quad  K_aK_a^{-1}=K_a^{-1}K_a=1, \  a=1,2,   \\
	&K_1 EK_1^{-1} =qE, \quad K_2EK_2^{-1}=-q E,\\
	&K_1 FK_1^{-1}=q^{-1}F, \quad K_2FK_2^{-1}=-q^{-1}F, \\
	&EF-FE= \frac{K_1K_2^{-1}- K_1^{-1}K_2}{q-q^{-1}},\\
	& E^2=F^2=0.
	\end{align*}
	Then using the element $\Gamma_V$ defined by \eqnref{eq: gammaV} and \thmref{thm-gen-center}, we obtain a central element
	\[
	C_1= q^{-1}K_1^2- q^{-1}K_2^2 -(q-q^{-1})^2 K_1K_2FE.
	\]
	This is analogous to the quantum Casimir element of  ${\rm U}_q(\mathfrak{gl}_{2})$.
\end{example}

To obtain a more concise formula for the central elements, we introduce a variant of $\Gamma_V$  defined by
\[
\Gamma:=\frac{\Gamma_V-1\otimes I}{q-q^{-1}}.
\]
This still satisfies the commutation relation in \lemref{lem: GammaV} and will be used to construct the central elements as described below.

We introduce the following elements:
\begin{equation}\label{eq: X}
\begin{aligned}
&X_{ab}=K_bE_{ab},\quad X_{ba}=K_a\bar{E}_{ba},\quad a<b,\\
&X_{aa}=\frac{K_a-1}{q-q^{-1}}, \quad a\in {\bf I}_{m,n}.
\end{aligned}
\end{equation}
Then we can rewrite $L^+$ and $(L^{-})^{-1}$ defined by \eqnref{def-L+} and \eqnref{def-L--} as
\begin{align*}
&L^{+}=1\otimes I+(q-q^{-1})\sum\limits_{a\le b}X_{ab}\otimes e_{ba},\\
&(L^{-})^{-1}=1\otimes I+(q-q^{-1})\sum\limits_{a\ge b} X_{ab} \otimes e_{ba}.
\end{align*}
It follows that
\[
\begin{aligned}
&\Gamma=\frac{(L^{-})^{-1}L^{+}-1\otimes I}{q-q^{-1}}
&=\sum\limits_{a, b\in {\bf I}_{m,n}}X_{ab}\otimes e_{ba}+(q-q^{-1})\sum\limits_{a\ge b, a\geq c}X_{ab} X_{ca}\otimes e_{bc}.
\end{aligned}
\]

The following is a consequence of \lemref{lem: GammaV} and \thmref{thm-gen-center}.
\begin{thm}\label{thm-center}
	For $k\in\Z_+$, the elements
	$$C_k={\rm Tr}_2\bigg((1\otimes \pi(K_{2\rho}))\big (\sum\limits_{a, b\in {\bf I}_{m,n}}X_{ab}\otimes e_{ba}+(q-q^{-1})\sum\limits_{a\ge b, a\geq c}X_{ab} X_{ca}\otimes e_{bc}\big )^k\bigg)$$
	lie in the centre of ${\rm U}_{q}(\gl_{m,n})$, where $\pi$ is the natural representation defined by \eqnref{eq: nat}, and  $K_{2\rho}$ and $X_{ab}$ are  defined by \eqref{eq: K2rho} and \eqnref{eq: X}, respectively.
\end{thm}

Using \thmref{thm-center}, it is straightforward to check that the central element $C_1$ associated to $\Gamma$ is given by the following formula: if $m+n$ is even,
\begin{equation}\label{eq: Cas1}
\begin{aligned}
C_1=&\sum_{a=1}^m q_a^{m-n+1-2a}\frac{K_a-1}{q-q^{-1}} + \sum_{b=m+1}^{m+n} q_{b}^{3m+n+1-2b}\frac{K_b-1}{q-q^{-1}}  \\
&+(q-q^{-1})\big(\sum_{a\geq b, b\leq m}q_{b}^{m-n+1-2b} K_{b} \bar{E}_{ab} K_a E_{ba}+\sum_{a\geq b, b>m}q_{b}^{3m+n+1-2b} K_{b} \bar{E}_{ab} K_a E_{ba}  \big),
\end{aligned}
\end{equation}
and if $m+n$ is odd, then
\begin{equation}\label{eq: Cas2}
\begin{aligned}
C_1=&\sum_{a=1}^m q_a^{m-n+2-2a}\frac{K_a-1}{q-q^{-1}} + \sum_{b=m+1}^{m+n} q_{b}^{3m+n-2b}\frac{K_b-1}{q-q^{-1}}  \\
&+(q-q^{-1})\big(\sum_{a\geq b, b\leq m}q_{b}^{m-n+2-2b} K_{b} \bar{E}_{ab} K_a E_{ba}+\sum_{a\geq b, b>m}q_{b}^{3m+n-2b} K_{b} \bar{E}_{ab} K_a E_{ba}  \big).
\end{aligned}
\end{equation}
We call element $C_1$ the \emph{quantum Casimir element} of ${\rm U}_q(\mathfrak{gl}_{m,n})$.

We conjecture that the elements $C_1, C_2, \dots, C_{m+n-1}$ 
generate the centre of $\mathrm{U}_q(\mathfrak{sl}_{m,n})$.  An analogous conjecture was made in \cite{Z2} for the quantum supergroup $\mathrm{U}_q(\mathfrak{sl}_{m|n})$, and remains open except when $n=0$, in which case \cite{L} shows the centre of $\mathrm{U}_q(\mathfrak{sl}_m)$ is generated by these analogous elements.

\section{A universal \texorpdfstring{$L$}{L}-operator}\label{sec: Lop}

The aim of this section is  to construct a spectral parameter-dependent solution $L(x)$ which satisfies the quantum Yang-Baxter equation in ${\rm U}_{q}(\mathfrak{gl}_{m,n})\otimes {\rm End}_{\mathcal{K}}(V)\otimes {\rm End}_{\mathcal{K}}(V)$.  Moreover, this approach yields an RLL realisation of the degenerate quantum general linear group, which we describe in \secref{sec: RLL}.

\subsection{The main result} Recall that the $R$-matrix associated to the natural representation $V$ of ${\rm U}_q(\mathfrak{gl}_{m,n})$ is an invertible element $R\in {\rm End}_{\mathcal{K}}(V\otimes V)$ satisfying the quantum Yang-Baxter equation in ${\rm End}_{\mathcal{K}}(V\otimes V\otimes V)$:
\[
R_{12}R_{13}R_{23}=R_{23}R_{13}R_{12}.
\]
This element $R$ has been constructed explicitly in \cite{CWZ}.

Our construction of $L(x)$ will make use of the $R$-matrix and the $L$-operators. Recall from \cite{CWZ} that the $R$-matrix is given by 
\begin{equation}\label{eq: Rmat}
	R=I\otimes I+\sum\limits_{a\in {\bf I}_{m,n}}(q_a-1)e_{aa}\otimes e_{aa}+(q-q^{-1})\sum\limits_{a<b}e_{ab}\otimes e_{ba}\in {\rm End}_{\mathcal{K}}(V\otimes V).
\end{equation}
Let $v_{a}, a\in {\bf I}_{m,n}$ be the standard basis of $V$. Then it is clear that
\begin{equation}
R(v_a\otimes v_b)= \begin{cases}
v_a\otimes v_b, &\quad a<b,\\
q_a v_a\otimes v_a, &\quad a=b, \\
v_a\otimes v_b+(q-q^{-1}) v_b\otimes v_a, &\quad a>b.
\end{cases}
\end{equation}
Let $T$ be the linear permutation operator on $V\otimes V$  such that $T(u\otimes v)=v\otimes u$ for $u,v\in V$, and let $R^{-T}:=T(R^{-1})T$. Then we have
\[
R^{-T}=I\otimes I+\sum\limits_{a\in {\bf I}_{m,n}}(q_a^{-1}-1)e_{aa}\otimes e_{aa}-(q-q^{-1})\sum\limits_{a<b}e_{ba}\otimes e_{ab}.
\]
For any   $x\in \mathbb{C}^{\ast}:= \mathbb{C}\backslash \{0\}$, we define the $R$-matrix $R(x)\in {\rm End}_{\mathcal{K}}(V\otimes V)$  by
\[
R(x):= xR- x^{-1}R^{-T}.
\]
Similarly, recalling the elements $L^{\pm}$ given in \eqref{def-L+} and \eqref{def-Lmin},  we introduce
\begin{equation}\label{eq: Lx}
L(x):= xL^{+}-x^{-1}L^{-}, \quad x\in \mathbb{C}^*.
\end{equation}
 The element $L(x)$ belongs to ${\rm U}_{q}(\mathfrak{gl}_{m,n})\otimes {\rm End}_{\mathcal{K}}(V)$ and is referred to as a \emph{universal $L$-operator} \cite{Z2}. Note that $(\pi\otimes 1) L(x) = R(x)$.

The following is the main result of this section.
\begin{thm}\label{thm-1}
	Let $L(x)$ be as  defined in \eqref{eq: Lx}. Then $L(x)$ satisfies the quantum
	Yang-Baxter equation in ${\rm U}_{q}(\gl_{m,n})\otimes \End_{\mathcal{K}} (V)\otimes \End_{\mathcal{K}} (V)$:
	\begin{equation*}\label{Y-B eq}
	L_{12}(x)L_{13}(xy)R_{23}(y)=R_{23}(y)L_{13}(xy)L_{12}(x),\quad x,y\in\C^*.
	\end{equation*}
\end{thm}

\subsection{Proof of \thmref{thm-1}}
To prove \thmref{thm-1}, we expand the quantum Yang-Baxter equation involving the parameters $x$ and $y$. Upon comparing the coefficient of each monomial term $x^iy^j$, it becomes evident that we only need to prove the following proposition.

\begin{prop}\label{prop: keyeq}
	Maintain the notation above. We have
	\begin{align*}
		L_{12}^{+}L_{13}^{+}R_{23}=R_{23}L_{13}^{+}L_{12}^{+}, \quad L_{12}^{-}L_{13}^{-}R_{23}=R_{23}L_{13}^{-}L_{12}^{-}, \quad
		L_{12}^{-}L_{13}^{+}R_{23}=R_{23}L_{13}^{+}L_{12}^{-}.
	\end{align*}
\end{prop}
\begin{proof}
We illustrate our method by providing a detailed proof of the first relation, as the remaining cases are similar. It suffices to prove $L_{12}^{+}L_{13}^{+}R_{23}=R_{23}L_{13}^{+}L_{12}^{+}$ by acting on the vectors $1\otimes v_c\otimes v_d$ ($c,d\in {\bf I}_{m,n}$).   The proof is then divided into the following three claims.
	
	\noindent
	{\bf Claim 1}: $L_{12}^{+}L_{13}^{+}R_{23}(1\otimes v_c\otimes v_c)=R_{23}L_{13}^{+}L_{12}^{+}(1\otimes v_c\otimes v_c)$ for all $c\in {\bf I}_{m,n}$.
	
	On the right hand side, we have
	\begin{align*}
	&R_{23}L_{13}^{+}L_{12}^{+}(1\otimes v_c\otimes v_c)
	=R_{23}L_{13}^{+}(K_c\otimes v_c\otimes v_c+(q-q^{-1})\sum\limits_{c<b}K_bE_{cb}\otimes v_b\otimes v_c)\\
	={}&R_{23}( K_c^2\otimes v_c\otimes v_c +(q-q^{-1})\sum\limits_{c<b}K_cK_bE_{cb}\otimes v_b\otimes v_c\\
	&+(q-q^{-1})\sum\limits_{c<b}K_bE_{cb}K_c\otimes v_c\otimes v_b+(q-q^{-1})^2\sum\limits_{c<b,c<a}K_bE_{cb}K_aE_{ca}\otimes v_a\otimes v_b)\\
	={} & q_cK_c^2\otimes v_c\otimes v_c +(q-q^{-1})\sum\limits_{c<b}K_cK_bE_{cb}\otimes v_b\otimes v_c\\
	&+(q-q^{-1})^2\sum\limits_{c<b}K_cK_bE_{cb}\otimes v_c\otimes v_b+(q-q^{-1})\sum\limits_{c<b}K_bE_{cb}K_c\otimes v_c\otimes v_b\\
	&+(q-q^{-1})^2\sum\limits_{c<a}q_aK_aE_{ca}K_aE_{ca}\otimes v_a\otimes v_a+(q-q^{-1})^3\sum\limits_{c<a<b}K_aE_{ca}K_bE_{cb}\otimes v_a\otimes v_b\\
	&+(q-q^{-1})^2(\sum\limits_{c<b<a}+\sum\limits_{c<a<b})K_bE_{cb}K_aE_{ca}\otimes v_a\otimes v_b.
	\end{align*}

	On the left hand side, we have
	\begin{align*}
	&L_{12}^{+}L_{13}^{+}R_{23}(1\otimes v_c\otimes v_c)= q_c L_{12}^{+}L_{13}^{+}(1\otimes v_c\otimes v_c)\\
	={}& q_c L_{12}^{+}(K_c\otimes v_c\otimes v_c+(q-q^{-1})\sum\limits_{c<b}K_bE_{cb}K_c\otimes v_c\otimes v_b)\\
	={} & q_cK_c^2\otimes v_c\otimes v_c +(q-q^{-1})q_c\sum\limits_{c<b}K_bE_{cb}K_c\otimes v_b\otimes v_c +(q-q^{-1})q_c\sum\limits_{c<b}K_cK_bE_{cb}\otimes v_c\otimes v_b\\
	&+(q-q^{-1})^2q_c\sum\limits_{c<a}K_aE_{ca}K_aE_{ca}\otimes v_a\otimes v_a+(q-q^{-1})^2q_c\sum\limits_{c<a<b}K_aE_{ca}K_bE_{cb}\otimes v_a\otimes v_b \\
	&+(q-q^{-1})^2q_c\sum\limits_{c<b<a}K_aE_{ca}K_bE_{cb}\otimes v_a\otimes v_b.
	\end{align*}
	These six terms in the above expression are denoted by  $S_i$,  $1\leq i\leq 6$, respectively.  Applying \lemref{lem: KErel}, we obtain
	\begin{align*}
	S_2&= (q-q^{-1}) \sum\limits_{c<b}K_bK_cE_{cb}\otimes v_b\otimes v_c,\\
	S_3&= (q-q^{-1})(q-q^{-1}+ q_{c}^{-1})\sum\limits_{c<b}K_cK_bE_{cb}\otimes v_c\otimes v_b\\
	&= (q-q^{-1})^2 \sum\limits_{c<b}K_cK_bE_{cb}\otimes v_c\otimes v_b + (q-q^{-1}) \sum\limits_{c<b}K_bE_{cb}K_c\otimes v_c\otimes v_b,
	\end{align*}
	where in the expression for $S_3$ we have used the fact that $q-q^{-1}= q_c-q_c^{-1}$ for any $c\in {\bf I}_{m,n}$. Let us consider $S_4$. If $c\geq m$, then $q_c=q_a$ for any $a>c$. If $a\geq m>c$, then $E_{ac}^2=0$ by \lemref{lem: Erel}. Therefore, we have
	\[
	\begin{aligned}
	S_4 &=(q-q^{-1})^2 \sum_{m\leq c<a} q_a K_aE_{ca}K_aE_{ca}\otimes v_a\otimes v_a + (q-q^{-1})^2q_c \sum_{c< m \leq a} q_{a}^{-1} K_aE_{ca}^2K_a \otimes v_a\otimes v_a  \\
	& = \sum_{c<a} (q-q^{-1})^2q_a K_aE_{ca}K_aE_{ca}\otimes v_a\otimes v_a.
	\end{aligned}
	\]
	Applying \lemref{lem: Erel} again, we obtain
	\begin{align*}
	S_5& = (q-q^{-1})^2(q_c^{-1}+ q-q^{-1})\sum\limits_{c<a<b}K_aE_{ca}K_bE_{cb}\otimes v_a\otimes v_b\\
	&= (q-q^{-1})^2 \sum\limits_{c<a<b}K_bE_{cb} K_aE_{ca}\otimes v_a\otimes v_b + (q-q^{-1})^3\sum\limits_{c<a<b} K_aE_{ca}K_bE_{cb}\otimes v_a\otimes v_b, \\
	S_6&= (q-q^{-1})^2 \sum\limits_{c<b<a}K_bE_{cb}K_aE_{ca}\otimes v_a\otimes v_b .
	\end{align*}
	Adding up $S_i$ for $1\leq i \leq 6$, we obtain the expression for $R_{23}L_{13}^{+}L_{12}^{+}(1\otimes v_c\otimes v_c)$ obtained earlier.

	\noindent
	{\bf Claim 2}: $L_{12}^{+}L_{13}^{+}R_{23}(1\otimes v_c\otimes v_d)=R_{23}L_{13}^{+}L_{12}^{+}(1\otimes v_c\otimes v_d)$ for all $c<d$.
	
	On the right hand side, we have
	\begin{align*}
	&R_{23}L_{13}^{+}L_{12}^{+}(1\otimes v_c\otimes v_d) \\
	=& R_{23} \big(K_dK_c\otimes v_c\otimes v_d + (q-q^{-1}) \sum_{d<b} K_b E_{db}K_c \otimes v_c \otimes v_b \\
	& +(q-q^{-1}) \sum_{c<b}  K_d K_b E_{cb} \otimes v_b \otimes v_d +(q-q^{-1})^2 \sum_{d<b,\ c<a} K_b E_{db} K_a E_{ca} \otimes v_a \otimes v_b \big) \\
	=&  K_dK_c\otimes v_c\otimes v_d + (q-q^{-1}) \sum_{d<b} K_b E_{db}K_c \otimes v_c \otimes v_b+  (q-q^{-1}) q_d K_d^2 E_{cd} \otimes v_d\otimes v_d \\
	& + (q-q^{-1}) (\sum_{c<b<d}+ \sum_{b>d}) K_d K_b E_{cb} \otimes v_b \otimes v_d +(q-q^{-1})^2 \sum_{b>d} K_d K_b E_{cb} \otimes v_d \otimes v_b \\
	&+ (q-q^{-1})^2 \sum_{b>d} q_b K_b E_{db}K_b E_{cb} \otimes v_b \otimes v_b + (q-q^{-1})^2 \sum_{b>d} K_b E_{db}K_d E_{cd} \otimes v_d \otimes v_b\\
	&+ (q-q^{-1})^2 (\sum_{d<a<b} +\sum_{c<a<d<b}) K_b E_{db} K_a E_{ca} \otimes v_a \otimes v_b \\
	& +(q-q^{-1})^2 \sum_{d<b<a} K_b E_{db} K_a E_{ca} \otimes v_a \otimes v_b + (q-q^{-1})^3 \sum_{d<b<a} K_b E_{db} K_a E_{ca} \otimes v_b \otimes v_a,
	\end{align*}
	where in the last equation we have partitioned the sum $\sum_{c<b}$ into $\sum_{b=d} + \sum_{c<b<d} +\sum_{b>d}$ and similarly split $\sum_{d<b, c<a}$  into $\sum_{d<a=b}+ \sum_{a=d<b} + \sum_{d<a<b}+ \sum_{c<a<d<b}+ \sum_{d<b<a}$.

	Considering the left hand side, we have
	\begin{align*}
	&L_{12}^{+}L_{13}^{+}R_{23}(1\otimes v_c\otimes v_d)=  L_{12}^{+}L_{13}^{+}(1\otimes v_c\otimes v_d)\\
	={}&  L_{12}^{+}(K_d\otimes v_c\otimes v_d+(q-q^{-1})\sum\limits_{b>d}K_bE_{db}K_c\otimes v_c\otimes v_b)\\
	={} & K_cK_d\otimes v_c\otimes v_d +(q-q^{-1})K_c\sum\limits_{b>d}K_bE_{db}\otimes v_c\otimes v_b+(q-q^{-1})\sum\limits_{b>c}K_bE_{cb}K_d\otimes v_b\otimes v_d\\
	&+(q-q^{-1})^2\sum\limits_{b>d}K_bE_{cb}K_bE_{db}\otimes v_b\otimes v_b+(q-q^{-1})^2\sum\limits_{b>d}K_dE_{cd}K_bE_{db}\otimes v_d\otimes v_b\\
	&+(q-q^{-1})^2(\sum\limits_{d<b<a}+\sum\limits_{c<a<d<b}+\sum\limits_{d<a<b})K_aE_{ca}K_bE_{db}\otimes v_a\otimes v_b.
	\end{align*}
	We denote the eight terms appearing in the above expression  by $S'_i, 1\leq i\leq 8$, respectively.
	
	Now we  compare two expressions for  $R_{23}L_{13}^{+}L_{12}^{+}(1\otimes v_c\otimes v_d)$ and $L_{12}^{+}L_{13}^{+}R_{23}(1\otimes v_c\otimes v_d)$. Notice that  $S'_1$ and $S'_2$ have appeared in the expression for $R_{23}L_{13}^{+}L_{12}^{+}(1\otimes v_c\otimes v_d)$. Splitting the sum $\sum_{b>c}$ into $\sum_{b=d} + \sum_{c<b<d}+ \sum_{b>d}$, we obtain
	\begin{align*}
	S'_3= (q-q^{-1})q_dK_d^2 E_{cd}\otimes v_d\otimes v_d +  (q-q^{-1}) (\sum_{c<b<d}+ \sum_{b>d}) K_d K_b E_{cb} \otimes v_b \otimes v_d .
	\end{align*}
	Using \lemref{lem: KErel} and \lemref{lem: Erel}, we have
	\begin{align*}
	S'_4& = (q-q^{-1})^2 \sum_{b>d} q_b K_b^2 E_{cb} E_{db} \otimes v_b \otimes v_b = (q-q^{-1})^2 \sum_{b>d} q_b^2 K_b^2 E_{db}  E_{cb} \otimes v_b \otimes v_b \\
	&=  (q-q^{-1})^2 \sum_{b>d} q_b K_b E_{db} K_b E_{cb} \otimes v_b \otimes v_b,\\
	S'_5 &=(q-q^{-1})^2 \sum_{b>d} K_dK_b E_{cd} E_{db} \otimes v_d \otimes v_b \\
	&= (q-q^{-1})^2 \sum_{b>d} K_dK_b (q_d^{-1} E_{db}E_{cd} +E_{cb}) \otimes v_d \otimes v_b \\
	&= (q-q^{-1})^2 \sum_{b>d} K_b E_{db} K_d E_{cd} \otimes v_d \otimes v_b + (q-q^{-1})^2\sum_{b>d} K_d K_b E_{cb} \otimes v_d\otimes v_b,\\
	S'_6 &= (q-q^{-1})^2 \sum\limits_{d<b<a}K_bE_{db}K_aE_{ca}\otimes v_a\otimes v_b,\\
	S'_7 &= (q-q^{-1})^2  \sum\limits_{c<a<d<b} K_bE_{db}K_aE_{ca}\otimes v_a\otimes v_b,\\
	S'_8 &= (q-q^{-1})^2 \sum_{d<a<b} K_a K_bE_{ca}E_{db} \otimes v_a \otimes v_b  \\
	&= (q-q^{-1})^2 \sum_{d<a<b} K_a K_b (E_{db} E_{ca}+(q-q^{-1})E_{da}E_{cb}) \otimes v_a\otimes v_b\\
	&= (q-q^{-1})^2 \sum_{d<a<b} K_a K_bE_{db} E_{ca} \otimes v_a\otimes v_b +  (q-q^{-1})^3 \sum_{d<a<b} K_a E_{da} K_b E_{cb} \otimes v_a\otimes v_b.
	\end{align*}
	Summing up the expressions for $S'_i, 1\leq i\leq 8$, we obtain  $R_{23}L_{13}^{+}L_{12}^{+}(1\otimes v_c\otimes v_d)$ as desired.

	Finally, we need to verify the following.
	
	\noindent
	{\bf Claim 3}: $L_{12}^{+}L_{13}^{+}R_{23}(1\otimes v_c\otimes v_d)=R_{23}L_{13}^{+}L_{12}^{+}(1\otimes v_c\otimes v_d)$ for all $c>d$.
	
	On the right hand side, we have
	\begin{align*}
	&R_{23}L_{13}^{+}L_{12}^{+}(1\otimes v_c\otimes v_d)= R_{23}L_{13}^+ (K_c\otimes v_c \otimes v_d + (q-q^{-1})\sum_{c<b} K_b E_{cb}\otimes v_b \otimes v_d ) \\
	=& R_{23}\big(
	K_dK_c \otimes v_c \otimes v_d +(q-q^{-1})\sum_{d<b}K_bE_{db} K_c \otimes v_c\otimes v_b
	+(q-q^{-1})\sum_{c<b} K_dK_bE_{cb} \otimes v_b \otimes v_d\\
	& + (q-q^{-1})^2 \sum_{d<b, c<a} K_bE_{db}K_aE_{ca}\otimes v_a\otimes v_b\big)\\
	=& K_dK_c\otimes v_c \otimes v_d + (q-q^{-1}) K_dK_c \otimes v_d\otimes v_c+ (q-q^{-1})q_c K_c E_{dc} K_c \otimes v_c\otimes v_c\\
	& + (q-q^{-1}) \sum_{b>c} K_b E_{db}K_c \otimes v_c \otimes v_b + (q-q^{-1}) \sum_{d<b<c} K_bE_{db}K_c \otimes v_c \otimes v_b\\
	&+ (q-q^{-1})^2 \sum_{d<b<c} K_b E_{db}K_c \otimes v_b \otimes v_c + (q-q^{-1})\sum_{b>c} K_d K_bE_{cb}\otimes v_b\otimes v_d \\
	&+ (q-q^{-1})^2 \sum_{b>c} K_d K_bE_{cb}\otimes v_d\otimes v_b +  (q-q^{-1})^2 \sum_{b>c} q_b K_bE_{db}K_bE_{cb} \otimes v_b \otimes v_b \\
	&+ (q-q^{-1})^2\sum_{a>c} K_c E_{dc}K_a E_{ca}\otimes v_a\otimes v_c + (q-q^{-1})^3 \sum_{a>c} K_c E_{dc}K_a E_{ca}\otimes v_c\otimes v_a \\
	&+ (q-q^{-1})^2 (\sum_{c<a<b}+\sum_{c<b<a}) K_bE_{db} K_a E_{ca} \otimes v_a\otimes v_b +(q-q^{-1})^3 \sum_{c<b<a} K_b E_{db}K_a E_{ca} \otimes v_b\otimes v_a \\
	&+ (q-q^{-1})^2 \sum_{d<b<c<a} K_bE_{db}K_a E_{ca}\otimes v_a\otimes v_b  + (q-q^{-1})^3 \sum_{d<b<c<a} K_bE_{db}K_a E_{ca}\otimes v_b\otimes v_a.
	\end{align*}
	
	We proceed to compute the left hand side:
	\begin{align*}
	&L_{12}^{+}L_{13}^{+}R_{23}(1\otimes v_c\otimes v_d)=  L_{12}^{+}L_{13}^{+}(1\otimes v_c\otimes v_d+(q-q^{-1})1\otimes v_d\otimes v_c)\\
	={}&  L_{12}^{+}(K_d\otimes v_c\otimes v_d+(q-q^{-1})K_c\otimes v_d\otimes v_c+(q-q^{-1})\sum\limits_{b>d}K_bE_{db}\otimes v_c\otimes v_b\\
	{}&+(q-q^{-1})^2\sum\limits_{b>c}K_bE_{cb}\otimes v_d\otimes v_b)\\
	={} & K_cK_d\otimes v_c\otimes v_d +(q-q^{-1})\sum\limits_{b>c}K_bE_{cb}K_d\otimes v_b\otimes v_d+(q-q^{-1})K_dK_c\otimes v_d\otimes v_c\\
	{}&+(q-q^{-1})^2\sum\limits_{b>d}K_bE_{db}K_c\otimes v_b\otimes v_c+(q-q^{-1})^2K_d\sum\limits_{b>c}K_bE_{cb}\otimes v_d\otimes v_b\\
	{}&+(q-q^{-1})K_c\sum\limits_{a>c,b>d}K_bE_{db}\otimes v_c\otimes v_b+(q-q^{-1})^2K_aE_{ca}\sum\limits_{b>d}K_bE_{db}\otimes v_a\otimes v_b\\
	{}&+(q-q^{-1})^3\sum\limits_{a>d,b>c}K_aE_{da}K_bE_{cb}\otimes v_a\otimes v_b\\
	={} & K_cK_d\otimes v_c\otimes v_d +(q-q^{-1})\sum\limits_{b>c}K_bE_{cb}K_d\otimes v_b\otimes v_d+(q-q^{-1})K_dK_c\otimes v_d\otimes v_c\\
	{}& +(q-q^{-1})^2K_cE_{dc}K_c\otimes v_c\otimes v_c+(q-q^{-1})^2(\sum\limits_{a>c}+\sum\limits_{d<a<c})K_aE_{da}K_c\otimes v_a\otimes v_c\\
	{}& +(q-q^{-1})^2\sum\limits_{a>c}K_dK_bE_{ca}\otimes v_d\otimes v_a+(q-q^{-1})K_c^2E_{dc}\otimes v_c\otimes v_c\\
	& + (q-q^{-1})(\sum\limits_{b>c}+ \sum_{d<b<c}  )K_cK_bE_{db}\otimes v_c\otimes v_b +(q-q^{-1})^2\sum\limits_{a>c}K_aE_{ca}K_aE_{da}\otimes v_a\otimes v_a\\
	&+ (q-q^{-1})^2\sum\limits_{a>c}K_aE_{ca}K_cE_{dc}\otimes v_a\otimes v_c +(q-q^{-1})^2(\sum\limits_{c<b<a}+\sum\limits_{c<a<b}+\sum\limits_{d<b<c<a})K_aE_{ca}K_bE_{db}\otimes v_a\otimes v_b\\
	&+(q-q^{-1})^3\sum\limits_{a>c}K_aE_{da}K_aE_{ca}\otimes v_a\otimes v_a + (q-q^{-1})^3 \sum_{b>c}K_c E_{dc}K_bE_{cb} \otimes v_c\otimes v_b\\
	{} &+(q-q^{-1})^3(\sum\limits_{c<a<b}+\sum_{c<b<a}+ \sum\limits_{d<a<c<b})K_aE_{da}K_bE_{cb}\otimes v_a\otimes v_b.
	\end{align*}

	It remains to compare the expressions for both sides. In the expression for $L_{12}^{+}L_{13}^{+}R_{23}(1\otimes v_c\otimes v_d)$, we observe
	\begin{align*}
	&(q-q^{-1})^2K_c E_{dc}K_c\otimes v_c \otimes v_c + (q-q^{-1})K_c^2E_{dc}\otimes v_c\otimes v_c\\
	=&(q-q^{-1})^2K_c E_{dc}K_c\otimes v_c \otimes v_c + (q-q^{-1})q_{c}^{-1}K_cE_{dc}K_c\otimes v_c\otimes v_c\\
	=& (q-q^{-1})q_cK_cE_{dc}K_c\otimes v_c\otimes v_c,
	\end{align*}
	where the last equation follows from the substitution $q_c^{-1}=q_c-( q-q^{-1})$. Note that the resulting expression is a term of $R_{23}L_{13}^{+}L_{12}^{+}(1\otimes v_c\otimes v_d)$. We continue the comparison. Applying \lemref{lem: KErel} and \lemref{lem: Erel}, we have
	\begin{align*}
	&(q-q^{-1})^2\sum\limits_{a>c}K_aE_{da}K_c\otimes v_a\otimes v_c+(q-q^{-1})^2\sum\limits_{a>c}K_aE_{ca}K_cE_{dc}\otimes v_a\otimes v_c\\
	={}& (q-q^{-1})^2\sum\limits_{a>c}K_a(E_{dc}E_{ca}-q_c^{-1}E_{ca}E_{dc})K_c\otimes v_a\otimes v_c+(q-q^{-1})^2q_c^{-1}\sum\limits_{a>c}K_aE_{ca}E_{dc}K_c\otimes v_a\otimes v_c\\
	={}& (q-q^{-1})^2\sum\limits_{a>c}K_aE_{dc}E_{ca}K_c\otimes v_a\otimes v_c\\
	={}& (q-q^{-1})^2\sum\limits_{a>c}K_cE_{dc}K_aE_{ca}\otimes v_a\otimes v_c.
	\end{align*}
	Using (\ref{eq:KE1}) and (\ref{eq:E12}), we deduce
	\begin{align*}
	& (q-q^{-1})^2(\sum\limits_{d<b<c<a}+\sum\limits_{b>a>c})K_aE_{ca}K_bE_{db}\otimes v_a\otimes v_b= (q-q^{-1})^2(\sum\limits_{d<b<c<a}+\sum\limits_{b>a>c})K_bE_{db}K_aE_{ca}\otimes v_a\otimes v_b.
	\end{align*}
	Using  (\ref{eq:KE1}) and (\ref{eq:E14}), we arrive at
	\begin{align*}
	& (q-q^{-1})^2\sum\limits_{a>b>c}K_aE_{ca}K_bE_{db}\otimes v_a\otimes v_b+(q-q^{-1})^3\sum\limits_{a>b>c}K_aE_{da}K_bE_{cb}\otimes v_a\otimes v_b\\
	={} & (q-q^{-1})^2\sum\limits_{a>b>c}(K_aK_bE_{ca}E_{db}+(q-q^{-1})K_aK_bE_{da}E_{cb})\otimes v_a\otimes v_b\\
	={} & (q-q^{-1})^2\sum\limits_{a>b>c}K_bK_a(E_{ca}E_{db}+(q-q^{-1})E_{da}E_{cb})\otimes v_a\otimes v_b\\
	={} & (q-q^{-1})^2\sum\limits_{a>b>c}K_bE_{db}K_aE_{ca}\otimes v_a\otimes v_b.
	\end{align*}
	From (\ref{eq:KE3}), (\ref{eq:E15}), and $q-q^{-1}+q_a^{-1}=q_a$ for $a\in {\bf I}_{m,n}$, we have
	\begin{align*}
	& (q-q^{-1})^2\sum\limits_{a>c}K_aE_{ca}K_aE_{da}\otimes v_a\otimes v_a+(q-q^{-1})^3\sum\limits_{a>c}K_aE_{da}K_aE_{ca}\otimes v_a\otimes v_a\\
	={}& (q-q^{-1})^2\sum\limits_{a>c}K_a(q_a^{-1}E_{ca}E_{da}K_a+(q-q^{-1})E_{da}K_aE_{ca})\otimes v_a\otimes v_a\\
	={}& (q-q^{-1})^2\sum\limits_{a>c}K_a(q_a^{-1}E_{da}K_aE_{ca}+(q-q^{-1})E_{da}K_aE_{ca})\otimes v_a\otimes v_a\\
	={}& (q-q^{-1})^2\sum\limits_{a>c}(q_a-1)K_aE_{da}K_aE_{ca}\otimes v_a\otimes v_a+(q-q^{-1})^2\sum\limits_{a>c}K_aE_{da}K_aE_{ca}\otimes v_a\otimes v_a.
	\end{align*}
	Now the remaining terms in  $L_{12}^{+}L_{13}^{+}R_{23}(1\otimes v_c\otimes v_d)$ and $R_{23}L_{13}^{+}L_{12}^{+}(1\otimes v_c\otimes v_d)$ coincide.
	Combining the results together, we complete the proof of Claim 3.
\end{proof}

We are in a position to prove \thmref{thm-1}.

\begin{proof}[Proof of \thmref{thm-1}]
Expanding the quantum Yang-Baxter equation from the theorem and comparing coefficients of $x^iy^j$ terms for  integers $i, j$, we obtain a system of equations: 
\begin{align}
	&L_{12}^{+}L_{13}^{+}R_{23}=R_{23}L_{13}^{+}L_{12}^{+},\label{eq-LLR1}\\ &L_{12}^{-}L_{13}^{-}R_{23}=R_{23}L_{13}^{-}L_{12}^{-},\label{eq-LLR2} \\
	&L_{12}^{-}L_{13}^{+}R_{23}=R_{23}L_{13}^{+}L_{12}^{-},\label{eq-LLR3}\\
	&L_{12}^{+}L_{13}^{+}R_{23}^{-T}=R_{23}^{-T}L_{13}^{+}L_{12}^{+}, \label{eq-LLR4}\\
	&L_{12}^{-}L_{13}^{-}R_{23}^{-T}=R_{23}^{-T}L_{13}^{-}L_{12}^{-},\label{eq-LLR5} \\
	&L_{12}^{+}L_{13}^{-}R_{23}^{-T}=R_{23}^{-T}L_{13}^{-}L_{12}^{+},\label{eq-LLR6}\\
	& L_{12}^{+}L_{13}^{-}R_{23}-L_{12}^{-}L_{13}^{+}R_{23}^{-T}=R_{23}L_{13}^{-}L_{12}^{+}-R_{23}^{-T}L_{13}^{+}L_{12}^{-}.\label{eq-LLR7}
\end{align}

The equations (\ref{eq-LLR1})-(\ref{eq-LLR3}) are given in \propref{prop: keyeq}. We will show that the remaining equations  follow from (\ref{eq-LLR1})-(\ref{eq-LLR3}). Denote by $T_{23}= 1\otimes T$ the linear map on $V\otimes V\otimes V$ that swaps the second and third tensor factors. We rewrite (\ref{eq-LLR1}) as  $R^{-1}_{23}L^+_{12}L^+_{13}=L^{+}_{13}L^+_{12} R^{-1}_{23}$. Conjugating both sides by $T_{23}$  yields \eqref{eq-LLR4}.  In exactly the same way, one obtains \eqref{eq-LLR5} and \eqref{eq-LLR6} by conjugating the rewritten forms of \eqref{eq-LLR2} and \eqref{eq-LLR3}, respectively.

It remains to prove \eqref{eq-LLR7}.  For the $R$-matrix given in \eqnref{eq: Rmat}, we define 
$\check{R}:=TR.$
Then we have $\check{R}=R^TT$ and hence $\check{R}^{-1}=TR^{-T}$. Also, by \cite[Corollary 4.9]{CWZ}, we have the Heckle relation
\begin{equation}\label{R-pro2}
	\check{R}^{-1}=R-(q-q^{-1}).
\end{equation}
Pre-multiplying both sides of (\ref{eq-LLR3}) by $T_{23}$, we have 
\begin{equation}\label{eq-LLR8}
	L_{13}^{-}L_{12}^{+}\check{R}_{23}=\check{R}_{23}L_{13}^{+}L_{12}^{-},
\end{equation}
and hence $\check{R}_{23}^{-1}L_{13}^{-}L_{12}^{+}=L_{13}^{+}L_{12}^{-}\check{R}_{23}^{-1}$. Now pre‑multiplying both sides of \eqref{eq-LLR7} by $T_{23}$ yields

\begin{align*}
	\mathrm{LHS}&=L_{13}^{+}L_{12}^{-}\check{R}_{23}-L_{13}^{-}L_{12}^{+}\check{R}_{23}^{-1}=L_{13}^{+}L_{12}^{-}(\check{R}_{23}^{-1}+q-q^{-1})-L_{13}^{-}L_{12}^{+}(\check{R}_{23}-(q-q^{-1}))\\
	&=\check{R}_{23}^{-1}L_{13}^{-}L_{12}^{+}+(q-q^{-1})L_{13}^{+}L_{12}^{-}-\check{R}_{23}L_{13}^{+}L_{12}^{-}+(q-q^{-1})L_{13}^{-}L_{12}^{+},\\ 
   \mathrm{RHS}&=\check{R}_{23}L_{13}^{-}L_{12}^{+}-\check{R}_{23}^{-1}L_{13}^{+}L_{12}^{-}=(\check{R}_{23}^{-1}+q-q^{-1})L_{13}^{-}L_{12}^{+}-(\check{R}_{23}-(q-q^{-1}))L_{13}^{+}L_{12}^{-}.
\end{align*}
Since $\mathrm{LHS}=\mathrm{RHS}$ and pre-multiplication by the invertible operator $T_{23}$ preserves equality, the original identity \eqref{eq-LLR7} follows.
\end{proof}

\section{The RLL realisation}\label{sec: RLL}
Motivated by \propref{prop: keyeq}, we give the RLL realisation of  the degenerate quantum group ${\rm U}_q(\mathfrak{gl}_{m,n})$. We refer to \cite{RTF} for  the RLL realisation of quantum groups. 
 
\begin{de}
Let $R$ be the $R$-matrix given by \eqnref{eq: Rmat}. 
	Let $\mathcal{U}(R)$ denote  the associative algebra over $\mathcal{K}$ generated by $\ell_{ab}^{+},\ell_{ba}^-,a\le b, a,b\in {\bf I}_{m,n}$, subject to the following relations: 
	\begin{align}\label{eq: RLL}
		&\mathcal{L}_{12}^{\pm }\mathcal{L}_{13}^{\pm } R_{23}= R_{23}\mathcal{L}_{13}^{\pm }\mathcal{L}_{12}^{\pm }, \quad \mathcal{L}_{12}^{-}\mathcal{L}_{13}^{+} R_{23}= R_{23}\mathcal{L}_{13}^{+}\mathcal{L}_{12}^{-},\\ 
         & \ell_{aa}^+ \ell_{aa}^-= \ell_{aa}^-\ell_{aa}^+=1, \quad a\in {\bf I}_{m,n}, 
	\end{align}
	where in \eqnref{eq: RLL} we set $\ell^+_{ab}=\ell^-_{ba}=0$ for $a>b$,    
\[ \mathcal{L}_{12}^{+}:= \sum_{a, b\in {\bf I}_{m,n}} \ell_{ab}^+\otimes e_{ba}\otimes 1, \quad \mathcal{L}_{12}^{-}:=\sum_{a, b\in {\bf I}_{m,n}}\ell_{ba}^-\otimes e_{ab}\otimes  1, \]
and   $\mathcal{L}_{13}^{\pm}$ are defined similarly by embedding $\sum_{a,b\in \mathbf{I}_{m,n}}\ell^{\pm}_{ab}\otimes e_{ba}$ into the the first and third tensor factors.   
\end{de}

\begin{prop}\label{prop: RTT}
The {\rm RLL} relations \eqnref{eq: RLL} among the generators $\ell^+_{ab}, \ell^-_{ba}, a\leq b, a,b\in {\bf I}_{m,n}, $ are given explicitly as follows. 
\begin{enumerate}
\item The equation $\mathcal{L}_{12}^{+}\mathcal{L}_{13}^{+} R_{23}= R_{23}\mathcal{L}_{13}^{+}\mathcal{L}_{12}^{+}$ is equivalent to the following relations: 
\begin{align*}
&\ell^+_{aa}\ell^+_{bb}= \ell^+_{bb}\ell^+_{aa}, \quad a,b \in {\bf I}_{m,n},\quad \quad (\ell^+_{ab})^2=0, \quad a\leq m<b,  \\ 
&\ell^+_{ab}\ell^+_{ac}= q_a \ell^+_{ac}\ell^+_{ab}, \quad b<c, \quad \quad \ell^+_{ab}\ell^+_{cb}= q_b \ell^+_{cb} \ell^+_{ab}, \quad a<c,\\ 
&\ell^+_{ab}\ell^+_{cd}= \ell^+_{cd}\ell^+_{ab}, \quad a<c, b>d, \\ 
&\ell^+_{ab}\ell^+_{cd}= \ell^+_{cd}\ell^+_{ab}+(q-q^{-1}) \ell^+_{cb}\ell^+_{ad}, \quad a<c, b<d. 
\end{align*}
\item The equation $\mathcal{L}_{12}^{-}\mathcal{L}_{13}^{-} R_{23}= R_{23}\mathcal{L}_{13}^{-}\mathcal{L}_{12}^{-}$ is equivalent to the following relations: 
\begin{align*}
&\ell^-_{aa}\ell^-_{bb}= \ell^-_{bb}\ell^-_{aa}, \quad a,b \in {\bf I}_{m,n},  \quad \quad  (\ell^-_{ab})^2=0, \quad a\leq m<b,  \\ 
&\ell^-_{ab}\ell^-_{ac}= q_a \ell^-_{ac}\ell^-_{ab}, \quad b<c, \quad \quad \ell^-_{ab}\ell^-_{cb}= q_b \ell^-_{cb} \ell^-_{ab}, \quad a<c,\\ 
&\ell^-_{ab}\ell^-_{cd}= \ell^-_{cd}\ell^-_{ab}, \quad a<c, b>d, \\ 
&\ell^-_{ab}\ell^-_{cd}= \ell^-_{cd}\ell^-_{ab}+(q-q^{-1}) \ell^-_{cb}\ell^-_{ad}, \quad a<c, b<d. 
\end{align*}
\item The equation $\mathcal{L}_{12}^{-}\mathcal{L}_{13}^{+} R_{23}= R_{23}\mathcal{L}_{13}^{+}\mathcal{L}_{12}^{-}$ is equivalent to the following relations: 
\begin{align*}
& \ell^-_{ab}\ell^+_{cb}= q_b \ell^+_{cb} \ell^-_{ab}, \quad a>c, \quad \quad\quad  \ell^-_{ab}\ell^+_{ac}= q_{a}^{-1} \ell^+_{ac}\ell^-_{ab}, \quad b<c,\\ 
& \ell^-_{ab}\ell^+_{cd}= \ell^+_{cd}\ell^-_{ab}, \quad a<c, b<d,\quad  \ell^-_{ab}\ell^+_{cd}= \ell^+_{cd}\ell^-_{ab},\quad   a>c, b>d,\\ 
& \ell^-_{ab} \ell^+_{cd}= \ell^+_{cd}\ell^-_{ab} +(q-q^{-1}) (\ell^+_{cb}\ell^-_{ad}- \ell^+_{ad}\ell^-_{cb}), \quad a>c, b<d.  
\end{align*}
\end{enumerate}
\end{prop}
\begin{proof}
Let $R=\sum_{a,b,c,d\in {\bf I}_{m,n}} R_{bd}^{ac} e_{ab}\otimes e_{cd}$. By \eqnref{eq: Rmat}, we have  
\begin{equation}\label{eq:R}
	R_{ab}^{ab}=1,\  a\neq b, \quad R_{aa}^{aa}=q_a, \quad R^{ab}_{ba}=q-q^{-1}, \ a<b, 
\end{equation}
and all other coefficients $R_{bd}^{ac}$ are zero. 
For part (1), note that 
\begin{align*}
	\mathcal{L}_{12}^{+}\mathcal{L}_{13}^{+} R_{23} &= \sum_{a,b,c,d,i,j,k,l\in {\bf I}_{m,n}}R_{bd}^{ac}( \ell_{ij}^+\otimes e_{ji}\otimes 1)(\ell_{kl}^+\otimes 1\otimes e_{lk}) (1\otimes e_{ab}\otimes e_{cd}) \\ 
	&= \sum_{i,j,k,l,b,d\in {\bf I}_{m,n}} R_{bd}^{ik} \ell^+_{ij}\ell^+_{kl}\otimes e_{jb}\otimes e_{ld}.      
\end{align*}
Similarly, we have 
\begin{align*} R_{23}\mathcal{L}_{13}^{+}\mathcal{L}_{12}^{+}
	&=\sum_{a,b,c,d,i,j,k,l\in {\bf I}_{m,n}}R_{bd}^{ac}( 1\otimes e_{ab}\otimes e_{cd})(\ell_{ij}^+\otimes 1\otimes e_{ji})( \ell_{kl}^+\otimes e_{lk}\otimes 1)  \\ 
	&= \sum_{i,j,k,l,a,c\in {\bf I}_{m,n}} R_{lj}^{ac} \ell^+_{ij}\ell^+_{kl}\otimes e_{ak}\otimes e_{ci}.  
\end{align*}
By the linear independence of the elements $e_{ab}\otimes e_{cd}, a,b,c,d\in {\bf I}_{m,n}$, we obtain 
\begin{equation}\label{eq:RT1}
	\sum_{i,k\in {\bf I}_{m,n}} R_{bd}^{ik} \ell^+_{ia}\ell^+_{kc} = \sum_{j,l\in {\bf I}_{m,n}} R_{lj}^{ac}  \ell^+_{dj}\ell^+_{bl}, \quad a,b,c,d\in {\bf I}_{m,n}. 
\end{equation}  	
Using \eqnref{eq:R}, we have 
\begin{align*}
	\sum_{i,k\in {\bf I}_{m,n}} R_{bd}^{ik} \ell^+_{ia}\ell^+_{kc}= \begin{cases}
		\ell^+_{ba}\ell^+_{dc}, &\text{if $b<d$,}\\
         q_b\ell^+_{ba}\ell^+_{bc}, &\text{if $b=d$,}\\
        \ell^{+}_{ba}\ell^+_{dc}+ (q-q^{-1})\ell^+_{da}\ell^+_{bc}, &\text{if $b>d$.}
	\end{cases} \\ 
	\sum_{j,l\in {\bf I}_{m,n}} R_{lj}^{ac}  \ell^+_{dj}\ell^+_{bl} = \begin{cases}
		\ell^+_{dc}\ell^+_{ba}+(q-q^{-1})\ell^+_{da}\ell^+_{bc}, &\text{if $a<c$,}\\
         q_a\ell^+_{da}\ell^+_{ba}, &\text{if $a=c$,}\\
        \ell^{+}_{dc}\ell^+_{ba}, &\text{if $a>c$.}
	\end{cases} 
\end{align*}
This yields nine commutation relations among the generators $\ell^+_{ab}, a,b\in {\bf I}_{m,n}$, with the convention that $\ell^+_{ab}=0$ whenever $a>b$. After removing redundancies  and reindexing appropriately, we obtain the relations stated in part (1). 

Similarly, in parts (2) and (3), the RLL relations can be rewritten as 
\begin{align*}\label{eq:RT2}
& \sum_{i,k\in {\bf I}_{m,n}} R_{bd}^{ik} \ell^-_{ia}\ell^-_{kc} = \sum_{j,l\in {\bf I}_{m,n}} R_{lj}^{ac}  \ell^-_{dj}\ell^-_{bl}, \quad a,b,c,d\in {\bf I}_{m,n}, \\ 
&\sum_{i,k\in {\bf I}_{m,n}} R_{bd}^{ik} \ell^-_{ia}\ell^+_{kc} = \sum_{j,l\in {\bf I}_{m,n}} R_{lj}^{ac}  \ell^+_{dj}\ell^-_{bl}, \quad a,b,c,d\in {\bf I}_{m,n}.
\end{align*}  
The explicit commutation  relations can be derived by using \eqnref{eq:R} and the convention that $\ell^{+}_{ab}=\ell^-_{ba}=0$ for $a>b$. 
\end{proof}

\begin{prop}
  The algebra $\mathcal{U}(R)$ is a Hopf algebra with comultiplication $\widetilde{\Delta}$, counit $\tilde{\epsilon}$, and antipode $\widetilde{S}$ determined by 
 \[
\widetilde{\Delta}(\ell_{ab}^{\pm}) =\sum_{c\in {\bf I}_{m,n}} \ell_{ac}^{\pm} \otimes \ell_{cb}^{\pm}, \quad \tilde{\epsilon}(\ell_{ab}^{\pm})=\delta_{ab},  \quad a,b \in {\bf I}_{m,n}, \quad \widetilde{S}(\mathcal{L}^{\pm})= (\mathcal{L}^{\pm})^{-1}, 
\]
 where $\mathcal{L}^{\pm}:= (\ell^{\pm}_{ab})$ are $(m+n)\times (m+n)$ matrices with $\ell_{ab}^+= \ell_{ba}^{-}=0$ whenever $a>b$. 
\end{prop}
\begin{proof}
It is straightforward to verify that $(\mathcal{U}(R),\widetilde{\Delta}, \tilde{\epsilon})$ is a bialgebra. Note that $\mathcal{L}^+$ is an upper triangular matrix with invertible diagonal entries $\mathcal{L}^+_{aa}$, hence it is invertible. Similarly, $\mathcal{L}^-$ is an invertible lower triangular matrix. We write the inverses $(\mathcal{L}^{\pm})^{-1}= (\tilde{\ell}^{\pm}_{ab})$. It is clear from \eqref{eq: RLL} that the inverses satisfy the opposite RLL relations in the opposite algebra of $\mathcal{U}(R)$.  Hence the map $\widetilde{S}: \mathcal{U}(R)\rightarrow \mathcal{U}(R)$ defined by $\widetilde{S}(\ell_{ab}^{\pm})= \tilde{\ell}_{ab}^{\pm}, a,b\in {\bf I}_{m,n}, $ is an algebra anti-morphism. It follows that $\widetilde{S}(\mathcal{L}^{\pm})\mathcal{L}^{\pm}=\mathcal{L}^{\pm}\widetilde{S}(\mathcal{L}^{\pm})=I$, where $I$ denotes the identity matrix.   This is equivalent to the antipode axiom $\sum_{c\in {\bf I}_{m,n}}\widetilde{S}(\ell_{ac})\ell_{cb}= \sum_{c\in {\bf I}_{m,n}}\ell_{ac}\widetilde{S}(\ell_{cb})= \epsilon(\ell_{ab})$. Therefore, $(\mathcal{U}(R), \widetilde{\Delta}, \tilde{\epsilon}, \widetilde{S})$ is a Hopf algebra. 
\end{proof}

Recall that  $\mathrm{U}_q(\mathfrak{gl}_{m,n})$ is a Hopf algebra with structure maps $\Delta, \epsilon, S$ defined by \eqnref{def-bialg} and \eqnref{eq: antipode}. It is well known that $(\mathrm{U}_q(\mathfrak{gl}_{m,n}), \Delta', \epsilon, S^{-1})$ also forms a Hopf algebra structure, where $\Delta'$ is the opposite coproduct, $S^{-1}$ denote the inverse of $S$, and $\epsilon$ remains the same.    

\begin{thm}
There exists a Hopf algebra isomorphism
\[ \psi: (\mathcal{U}(R), \widetilde{\Delta}, \tilde{\epsilon}, \widetilde{S}) \longrightarrow (\mathrm{U}_q(\mathfrak{gl}_{m,n}),\Delta', \epsilon, S^{-1}),  \]
 which is defined  on the generators by 
	\[
	\begin{aligned}
		&\psi(\ell_{aa}^{\pm})= K_a^{\pm 1}, \quad  
		\psi(\ell_{ab}^{+})= (q-q^{-1})K_b E_{ab}, \\ 
		 &\psi(\ell_{ba}^{-}) = -(q-q^{-1})E_{ba} K_b^{-1}, \quad a<b,\quad a,b\in {\bf I}_{m,n}.
	\end{aligned}
	\]	
\end{thm}
\begin{proof}  
We first prove that $\psi$ is an algebra isomorphism. 
By \propref{prop: keyeq}, the assignment above extends to a well‑defined algebra homomorphism $\psi: \mathcal{U}(R)\to \mathrm{U}_q(\mathfrak{gl}_{m,n})$. Since the image generates $\mathrm{U}_q(\mathfrak{gl}_{m,n})$,  $\psi$ is surjective.  

Conversely, define an algebra homomorphism  $\phi: \mathrm{U}_q(\mathfrak{gl}_{m,n})\rightarrow \mathcal{U}(R)$ on the Chevalley generators  by 
\begin{align*}
&\phi(K^{\pm 1}_b)=\ell^{\pm}_{bb},   \quad b\in \mathbf{I}_{m,n}, \quad \phi(e_a)= (q-q^{-1})^{-1} \ell^-_{a+1, a+1}\ell^+_{a,a+1},\\  
&\phi(f_a)= -(q-q^{-1})^{-1} \ell^-_{a+1,a}\ell^+_{a+1, a+1}, \quad a\in \mathbf{I}'_{m,n}.    
\end{align*}
The RLL relations from \propref{prop: RTT} ensure that $\phi$ respects all defining relations, so it is well‑defined.   Moreover,  similar to the proof of \lemref{S-act on E}, one shows by induction on $b-a$ that 
\[ \phi(E_{ab})=(q-q^{-1})^{-1}\ell^-_{bb}\ell^+_{ab}, \quad \phi(E_{ba})= -(q-q^{-1})^{-1}\ell^-_{ba}\ell^+_{bb}, \quad a<b.    \]
It follows that $\phi(\psi (\ell^{+}_{ab}))=\ell^{+}_{ab}$ and $\phi(\psi (\ell^{-}_{ba}))=\ell^{-}_{ba}$ for $a\leq b$.  Hence $\phi\psi=1_{\mathcal{U}(R)}$, so $\psi$ is injective and therefore  an algebra isomorphism. 
 
Finally, a direct calculation shows that 
\[
\widetilde{\Delta}\phi = (\phi\otimes \phi )\Delta', \quad \tilde{\epsilon} \phi = \epsilon, 
 \]
so $\phi$ is a bialgebra isomorphism. Since the antipode compatibility follows automatically, $\phi$  (and hence $\psi$) is a Hopf algebra isomorphism. 
\end{proof}

\begin{example}
	Consider ${\rm U}_q(\mathfrak{gl}_{1,1})$ from Example \ref{exam: simp} generated by $K_1^{\pm 1}, K_2^{\pm 1}$,  $E=E_{12}$ and $F=E_{21}$.  The algebra $\mathcal{U}(R)$ is generated by $\ell^{\pm}_{11}, \ell_{22}^{\pm}, \ell_{12}^{+}, \ell_{21}^{-}$ subject to the following relations: 
	\begin{align*}
	 &\ell^{\pm}_{11}\ell^{\pm}_{22}= \ell^{\pm}_{22}\ell^{\pm}_{11}, \quad \ell_{aa}^+ \ell_{aa}^-= \ell_{aa}^-  \ell_{aa}^+=1, \quad a=1,2, \\ 
	 & (\ell^+_{12})^2= (\ell^-_{21})^2=0, \\ 
	 &\ell^+_{11}\ell^+_{12}= q \ell^+_{12}\ell^+_{11}, \quad \ell^+_{12}\ell^+_{22}= -q^{-1} \ell^+_{22}\ell^+_{12}, \\& \ell^-_{11}\ell^-_{21}= q \ell^-_{21}\ell^-_{11}, \quad \ell^-_{21}\ell^-_{22}= -q^{-1} \ell^-_{22}\ell^-_{21}, \\
	 & \ell^-_{21} \ell^+_{12}= \ell^+_{12}\ell^-_{21}+(q-q^{-1})(\ell^+_{11}\ell^-_{22}- \ell^-_{11}\ell^+_{22}).
	\end{align*}
The Hopf algebra isomorphism between $\mathcal{U}(R)$ and $\mathrm{U}_q(\mathfrak{gl}_{1,1})$ is given explicitly by 
	\[
	\begin{aligned}
		&\ell^+_{11}\mapsto K_1,  \quad \ell^+_{12}\mapsto (q-q^{-1})K_2E, \quad \ell^+_{22}\mapsto K_2, \\ 
		&\ell^-_{11}\mapsto  K_1^{-1}, \quad \ell^-_{21}\mapsto -(q-q^{-1})F K_2^{-1}, \quad \ell^-_{22}\mapsto K_2^{-1}. \\ 
	\end{aligned}
	\]
\end{example}

\noindent{\bf Data availability statement}

\noindent Data sharing is not applicable to this article as no new data were created or analysed in this study.

\end{document}